\documentclass[11pt,letterpaper]{amsart}
\usepackage{amsfonts,amsmath,amssymb,color,esint}
\usepackage{enumerate}
\usepackage{graphicx}
\usepackage{tikz}
\usetikzlibrary{decorations.markings}
\usetikzlibrary{plotmarks}
\usetikzlibrary{patterns}
\numberwithin{equation}{section}

\usepackage{amsthm}

\usepackage[babel=true,kerning=true]{microtype}

\usepackage{amsthm,leqno}
\usepackage{hyperref}

\newtheorem{ques}{Question}
\newtheorem{theo}{Theorem}[section]

\newtheorem{defi}{Definition}[section]
\newtheorem{cl}{Claim}[section]
\newtheorem{prop}{Proposition}[section]

\newcommand{\eps}{\varepsilon}

\newcommand{\R}{\mathbb{R}}

\begin{document}

\title{Non planar free boundary minimal disks into ellipsoids}
\author{Romain Petrides}
\address{Romain Petrides, Universit\'e de Paris, Institut de Math\'ematiques de Jussieu - Paris Rive Gauche, b\^atiment Sophie Germain, 75205 PARIS Cedex 13, France}
\email{romain.petrides@imj-prg.fr}

\begin{abstract} 
We prove the existence of embedded non planar free boundary minimal disks into rotationally symmetric ellipsoids of $\R^3$. The construction relies on the optimization of combinations of first and second Steklov eigenvalues renormalized by the length of the boundary, among metrics on the disk. We also prove that non planar free boundary harmonic maps from a disk into an ellipsoid of $\R^3$ such that the coordinate functions are first or second Steklov eigenfunctions with respect to the associated critical metric is a minimal immersion (without branched points) and that if the critical metric is even with respect to the coordinates of the disk, then the minimal immersion is an embedding.
\end{abstract}

\maketitle

We consider the following 2-ellipsoid parametrized by $\sigma:= (\sigma_1,\sigma_2,\sigma_3)$:
$$ \mathcal{E}_{\sigma}:= \{ (x_1,x_2,x_3) \in \mathbb{R}^3 ; \sigma_1 x_1^2 + \sigma_2 x_2^2 + \sigma_3 x_3^2 = 1 \} $$
with semi-axes $\left(\sigma_i\right)^{-\frac{1}{2}}$ for $i=1,2,3$. The equatorial curves $\{x_i = 0\} \cap \mathcal{E}_{\sigma}$ for $i=1,2,3$ are simple closed geodesic in $\mathcal{E}_{\sigma}$. By a result by Morse \cite{morse} (see also \cite{klingenberg}), they are also the only ones for ellipsoids close to the sphere. In a celebrated result, proved by the combination of the min-max method by Lusternik and Schnirelmann \cite{LS47} and mean flow methods by Grayson \cite{grayson}, any closed Riemannian two-sphere contains at least 3 simple closed geodesics. It was a refinement of the min-max method initiated by Birkhoff \cite{birkhoff} to prove existence of simple closed geodesics in a Riemannian topological sphere. Many generic ellipsoids $\mathcal{E}_{\sigma}$ realize then the minimal number of simple closed geodesics $3$. Later, Viesel \cite{viesel} proved existence of ellipsoids that contain an arbitrary number of simple closed geodesics (see also \cite{klingenberg2}).

In the same spirit, we ask for embedded free boundary minimal disks into $\mathcal{E}_{\sigma}$. By definition, free boundary minimal disks into a surface $\Sigma \subset \R^3$ are topological disks that are critical points of the area functional among all disks $D$ such that $\partial D \subset \Sigma$. They are exactly the minimal disks $D$ of $\R^3$ such that $\partial D \subset \Sigma$ and $D$ meets orthogonally $\Sigma$ on $\partial D$. The planar disks $ \{ x_i = 0 \} \cap co\left(\mathcal{E}_{\sigma}\right)  $ are the first trivial examples of free boundary minimal embedded disks with $\Sigma = \mathcal{E}_{\sigma}$. In the current paper we are interested in the following question:

\begin{ques}[Dierkes, Hildebrandt, K\"uster, Wohlrab, 1993, \cite{DHKW} p335] \label{Q1}
Are there non-planar embedded free boundary minimal disks into ellipsoids of $\R^3$ ? 
\end{ques}

One analogous question by Yau in 1987 \cite{yau} was raised in the case of 2-spheres into 3-ellipsoids: are the only minimal 2-spheres in an ellipsoid centered about the origin in $\R^4$ the planar ones?  Haslhofer and Ketover recently proved  that the answer is no \cite{HK19} for sufficiently elongated ellipsoids, thanks to subtle min-max and mean curvature flow methods. Since then, other constructions of non planar minimal spheres in \cite{BP22} and \cite{Pet23} have been given by different methods. In \cite{BP22}, Bettiol and Piccione build an arbitrary number of non planar minimal spheres for sufficiently elongated rotationally symmetric ellipsoids by bifurcations methods. In \cite{Pet23}, we use analogous techniques as in the current paper with optimization of combination of eigenvalues of the Laplacian.

\medskip 

Notice that the free boundary minimal 2-disks into the boundary of a 3-d manifold can purely be seen as a 1-d problem into a 2-d surface in a non local setting: they are half-harmonic maps into the surface. 
More precisely, if it is classical that $f : \R^2_+ \to \R$ 
is a harmonic function, then $- \partial_y f = \Delta^{\frac{1}{2}} f$ on $\R \times \{0\}$. Following the definition by Da Lio and Rivi\`ere \cite{DR11}, if $\Sigma\subset \R^3$ is a surface, we say that $f : \R \to \Sigma$ is a half harmonic map into $\Sigma$ if it is a critical for the energy $\int_{\R} \left\vert \Delta^{\frac{1}{4}} f \right\vert^2$ under the constraint $f \in \Sigma$ a.e. The Euler-Lagrange equation reads as
\begin{equation*}  \Delta^{\frac{1}{2}} f  \perp T_{f} \Sigma \text{ on } \R\times \{0\}. \end{equation*}
We set $\Phi = \hat{f} \circ h : \mathbb{D} \to \R^3$, where $h : \mathbb{D} \to \mathbb{R}^2_+$ is the biholomorphism such that $h(1) = 0$, $h(0) = 1$ and $h(-1)=\infty$ and $\hat{f} : \R^2_+ \to \R^3 $ is the harmonic extension of $f$. Then $f$ is a half-harmonic map into $\Sigma$ if and only if $\Phi$ is a free boundary harmonic map into $\Sigma$ (defined as critical maps of $\int_{\mathbb{D}} \left\vert \nabla \Psi \right\vert^2$ under the constraint $\Psi(\mathbb{S}^1)\subset \Sigma$ a.e). In this case, the equation reads as:
\begin{equation*} 
\begin{cases}
\Delta \Phi = 0 \text{ in } \mathbb{D} \\
\partial_r \Phi \parallel \Delta^{\frac{1}{2}} f  \perp T_{\Phi} \Sigma \text{ on } \mathbb{S}^1.
 \end{cases} \end{equation*}
The restriction at the boundary then satisfies a non local version of the equation of closed geodesics on $\Sigma$.
By a classical argument on the Hopf differential, if $\Phi$ is a free boundary harmonic map, then it is a branched conformal immersion and then a free boundary minimal branched immersed disk into $\Sigma$. If $\Sigma = \mathcal{E}_\sigma$, there is a natural weaker question than the previous one: 

\begin{ques} \label{Q2}
Is there a non planar simple closed half harmonic maps $\Phi : \mathbb{S}^1 \to \mathcal{E}_\sigma$ ? 
\end{ques}

The answer of Question \ref{Q1} and \ref{Q2} is no if the ellipsoid is a round sphere ($\sigma_1 = \sigma_2 = \sigma_3$) by Nitsche \cite{nitsche}: the only free boundary minimal disks into spheres are the planar ones. This result is the same as the classical one for 2-spheres into 3-spheres by Almgren \cite{almgren}. In this context, free-boundary minimal disks are more rigid, because even if we increase the dimension of the target sphere, the only free boundary branched immersed minimal disks are the planar ones \cite{FS15}, while we have existence of many non planar minimal spheres into higher dimensional spheres (leading to a theory initiated by Calabi \cite{calabi} and Barbosa \cite{barbosa}). 

\medskip

The construction of free boundary minimal disks has been investigated by Struwe \cite{struwe}, Fraser \cite{fraser} (with methods inspired by Sacks and Uhlenbeck \cite{SU81}), Laurain, Petrides \cite{LP19} and Lin, Sun, Zhou \cite{LSZ20} (with min-max methods inspired the replacement or shortening process by Birkhoff and Colding and Minicozzi \cite{CM08}), Gruter-Jost \cite{GJ86} and Li-Zhou \cite{LZ21} (with min-max methods coming from geometric measure theory initiated by Almgren \cite{almgren2}). In all these constructions, the variational methods are not as refined as we need to obtain new \textit{embedded} free boundary minimal \textit{disks} into \textit{2-ellipsoids}. As in the closed case, the free boundary version \cite{wang} by Wang of the proof of the celebrated Yau conjecture by Song \cite{song} is not helpful to build new disks, since they do not give information on the topological type of their minimal surfaces.

The submentioned methods in the closed case: \cite{HK19} or \cite{BP22} may still be possible to address the free boundary case but involve new difficulties. For instance in \cite{BP22}, the authors spectacularly use a rotational invariance assumption to prove existence of an arbitrary number of embedded non-planar rotationally invariant minimal spheres into sufficiently ellongated rotationally symmetric 3-d ellipsoids with bifurcation methods. However, we cannot build any embedded non-planar rotationally invariant free boundary minimal disk into a rotationally symmetric 2-ellipsoid. 

\medskip

In the current paper, we shall use an indirect way: the shape optimization of combinations of Steklov eigenvalues. Indeed, in \cite{Pet21}, (see also \cite{Pet22} and \cite{PT23}) the author noticed that all possibly branched conformally immersed minimal map $\Phi$ from a 2-surface with boundary $\Sigma$ into a n-ellipsoid $\mathcal{E}$ can be seen as a shape associated to a critical metric of some functional depending on combinations of Steklov eigenvalues with respect to the metric. In addition, the $k$-th parameter of the ellipsoid is associated to a Steklov eigenvalue involved in the functional at the critical metric, and the $k$-th coordinate of the immersion is an eigenfunction associated to this eigenvalue. This is a generalization of a result by Fraser and Schoen when only one eigenvalue appears in the functional: the target manifold is then a ball. Their discovery led to a method for the construction of free boundary minimal surface of genus $0$ into 2-spheres \cite{FS16} extended in \cite{Pet19} for higher genuses and higher eigenvalues. Applying the variational methods developped for combinations of Steklov eigenvalues in \cite{Pet21} and \cite{Pet22}, we are in position to prove the following result:
\begin{theo} \label{theomain} For any $s \neq 0$, there is a one parameter family $(p_{s,t})_{t > t_{\star}(s)}$ for some $t_{\star}(s)>0$ such that there is an embedded free boundary non planar minimal topological disk $D_{s,t}$ into the rotationally symmetric ellipsoid 
$$\mathcal{E}_{\sigma_{s,t}} := \{ x\in \mathbb{R}^3 ; p_{s,t} x_0^2 + x_1^2 +  x_2^2  = 1 \}$$ 
and such that the coordinates $x_0 , x_2 , x_2$ are first and second Steklov eigenfunctions for an induced metric $g_{s,t} = e^{2v_t} \xi$ on $D_{s,t}$ such that
$$ e^{v_{s,t}} = \frac{1}{\left(p_{s,t}^2 x_0^2 +  x_1^2 +  x_2^2  \right)^{\frac{1}{2}}} \text{ on } \partial{D_{s,t}} \text{ and } L_{s,t} = \int_{\partial D_{s,t}}dL_{g_{s,t}}$$
where $\xi$ is the standard Euclidean metric. For any $s$ and $t_1 < t_2$, $D_{s,t_1}$ is not isometric to $D_{s,t_2}$. Moreover $t \mapsto L_{s,t}p_{s,t}$ is decreasing and $t \mapsto L_{s,t}$ is increasing, and we have that $p_{s,t} \to 0$ and $L_{s,t} \to 4\pi$ as $t\to +\infty$ and that $D_{s,t}$ converges as varifolds to the disk $\{0\} \times \mathbb{D}^2$ with multiplicity $2$ as $t\to +\infty$.
\end{theo}

This theorem gives a positive answer to Question \ref{Q1} and Question \ref{Q2}. We obtain $D_{s,t} = \Phi\left(\mathbb{D}\right)$ as the image of the minimal free boundary immersion $\Phi : \mathbb{D} \to \mathcal{E}_{\sigma_{s,t}}$ associated to a critical metric $g_{s,t} = e^{2v_{s,t}} (dx^2 + dy^2)$ for a minimization problem with respect to the functionals
$$ F_{s,t}: g \mapsto \left(\bar{\sigma}_1(g)^{-s} + t \bar{\sigma}_2(g)^{-s}\right)^{\frac{1}{s}} $$
for $t > 0 $ and $s \neq 0$, where for $k\geq 0$, $\bar{\sigma}_k(g) = \sigma_k(g) L_g(\mathbb{S}^1)$ denotes the renormalized $k$-th Steklov eigenvalue $\sigma_k(g)$ by the length $L_g(\mathbb{S}^1)$ of the boundary of $\mathbb{D}$ with respect to $g$. By convention, notice that $\sigma_0 = 0$ is the simple eigenvalue associated to constant functions because $\mathbb{D}$ is connected and $\sigma_1(g)>0$ is the first non-zero eigenvalue and that we can have $\sigma_2(g) = \sigma_1(g)$ if $\sigma_1(g)$ is multiple. Notice that we can also try many other functionnals depending on $\bar{\sigma}_1$ and $\bar{\sigma}_2$ than $F_{s,t} =: h_{s,t}(\bar{\sigma}_1, \bar{\sigma}_2 )$.

With the notations of Theorem \ref{theomain}, we have $\bar{\sigma}_1\left(D_{s,t},g_{s,t}\right) = p_{s,t}L_{s,t}$ and  $\bar{\sigma}_2\left(D_{s,t},g_{s,t}\right) = L_{s,t}$. Since $\bar{\sigma}_2(g_{s,t})\to 4\pi$, \ref{eqW3} shows that $g_{s,t}$ is a particular maximizing sequence for $\bar{\sigma}_2$ as $t\to +\infty$. Notice also that
\begin{equation*}
\begin{split} L_{s,t} = & \int_{\partial D_{s,t}}dL_{g_{s,t}} = \int_{\partial D_{s,t}} (p_{s,t}x_0^2 + x_1^2 + x_2^2)dL_{g_{s,t}} 
\\ = & \int_{D_{s,t}}\left( \left\vert \nabla x_0 \right\vert_{g_{s,t}}^2 + \left\vert \nabla x_1 \right\vert_{g_{s,t}}^2 + \left\vert \nabla x_2 \right\vert_{g_{s,t}}^2 \right)dA_{g_{s,t}} = 2A(D_{s,t})
\end{split}
\end{equation*}
where $A(D_{s,t})$ is the area of the minimal surface. As explained below, we emphasize that $\Phi : \mathbb{D}\to \mathcal{E}_{\sigma_{s,t}}$ is an embedding and its geometric properties (symmetries, intersection with equatorial planes) are described in Claim \ref{claimsymmetry}.

\medskip

Let's explain the steps of proof for Theorem \ref{theomain}. It is not \textit{a priori} clear that the following properties hold true:
\begin{itemize}
\item[\textbf{(1)}] $F_{s,t}$ has a minimizer, and then there is a minimal (possibly branched) free boundary immersion $\Phi$ from $\mathbb{D}$ into a $n$-ellipsoid associated to this minimal metric.
\item[\textbf{(2)}] Up to rearrangement of coordinates, $\Phi$ has at most $3$ coordinates and the target of $\Phi$ is the convex hull of a $2$-ellipsoid $\mathcal{E}_{\sigma_{s,t}}$ for $\sigma_{s,t} = (\sigma_1(g_{s,t}),\sigma_2(g_{s,t}),\sigma_2(g_{s,t}))$.
\item[\textbf{(3)}] $D_{s,t} = \Phi\left(\mathbb{D}\right)$ is non-planar.
\item[\textbf{(4)}] $\Phi$ is an embedding.
\end{itemize}
Let's explain every difficulty separately:

\medskip

\textbf{(1)}: For the choice $t=0$, we know by Weinstock inequality \cite{weinstock}
\begin{equation} \label{eqW1} \bar{\sigma}_1 \leq 2\pi \end{equation}
and that $2\pi$ is only realized for the flat disk. Weinstock arguments also give that for $t=1$ and $s=1$, 
\begin{equation} \label{eqW2} \frac{1}{\bar{\sigma}_1} + \frac{1}{\bar{\sigma}_2} \geq \frac{1}{\pi} \end{equation}
and the unique minimizer is the flat disk. For $t = \infty$: the only eigenvalue which appears is $\bar{\sigma}_2$, then, by Hersch Payne and Schiffer \cite{HPS75}
\begin{equation} \label{eqW3} \bar{\sigma}_2 < 4\pi  \end{equation}
and it is not realized (see \cite{GP10}). $4\pi$ corresponds to the disjoint union of two flat disks of same boundary length. In particular, there are maximizing sequences approaching $4\pi$ which blow up and "bubble converge" to a disjoint union of two disks (see for instance Theorem \ref{theotestfunction} below). In \cite{Pet21}, we give an optimal assumption on the functional to ensure that bubbling cannot occur and then to ensure existence of a minimizer. In our context, setting $h_{s,t}(\sigma_1,\sigma_2) = \left(\sigma_1^{-s} + t \sigma_2^{-s}\right)^{\frac{1}{s}} $, the condition reads as
\begin{equation}\label{eqinfcond} \inf_g F_{s,t} < h_{s,t}( 0, 4\pi ). \end{equation}
This condition is automatic if $s>0$, but needs to be discussed if $s<0$. In section \ref{spectralgaps}, we prove the following
\begin{theo} \label{theotestfunction} There is a one parameter family of metrics $h_{\eps} = e^{2v_\eps} (dx^2 +dy^2 )$ such that 
$$ \bar{\sigma}_1(h_\eps) = \frac{2\pi}{\ln\left(\frac{1}{\eps}\right)} + O\left( \frac{1}{\ln\left(\frac{1}{\eps}\right)^2}\right)  \text{ and }  \bar{\sigma}_2(g_\eps) = 4\pi - 16\pi \eps + o(\eps)  $$
as $\eps \to 0$ and $e^{2v_\eps}$ satisfies the following symmetry properties
$$ \forall (x,y) \in \mathbb{D}, e^{2v_\eps}(x,y) = e^{2v_\eps}(-x,y) = e^{2v_\eps}(x,-y).$$
\end{theo}
This theorem ensures that a minimum also exists for $F_{s,t}$ with $s<0$ and $t >0$.

\medskip

\textbf{(2)}: The coordinate functions of the minimal free boundary immersion $\Phi : \mathbb{D} \to \R^{n+1}$ into a $n$-ellipsoid are associated to first and second Steklov eigenvalues with respect to the minimal metric for $F_{s,t}$. Up to rearrangement of coordinates, we can assume that they are independent eigenfunctions. Knowing by Jammes \cite{jammes}, that the multiplicity of $\sigma_1$ and $\sigma_2$ on the disk is less than $2$, we obtain that $n+1=2$ or $n+1=3$. In all the cases, the target surface is the convex hull of $\mathcal{E}_{\sigma_{s,t}}$ with $\sigma_{s,t} = (\bar{\sigma}_1(g_{s,t}),\bar{\sigma}_2(g_{s,t}),\bar{\sigma}_2(g_{s,t}))$ where if $n+1=2$, we set by convention $\phi_3 = 0$ and the minimal immersion is planar.

\medskip

\textbf{(3)}: As noticed in \eqref{eqW1} and \eqref{eqW2}, for $t=0$ or for $s=1$ and $t \leq 1$ the only minimizers are flat disks, so that we do not build non planar minimal disks into ellipsoids in these cases. Then, we have to choose carefully the parameters $s$ and $t$ that ensure that minimizers cannot be planar ellipses. In the current paper (see Proposition \ref{propdiffeodim2}), we prove that any planar immersed disk into the convex hull of an ellipse has to be a diffeomorphism. Then if $\Phi\left( \mathbb{D} \right)$ is planar, there is only one candidate and the associated critical metric $g_c$ satisfies
$$ \bar{\sigma}_1(g_c) = 2\pi \sqrt{\frac{p}{q}} \text{ and } \bar{\sigma}_2(g_c) = 2\pi \sqrt{\frac{q}{p}} $$
where up to a dilatation, $(\mathbb{D},g_{c})$ is isometric to $E_{p,q} = \{ (x,y)\in \mathbb{R}^2 ; p x^2 + q y^2 = 1 \}$ endowed with a metric $e^{2v} (dx^2+dy^2)$ such that the conformal factor satisfies
$$\forall (x,y)\in \partial E_{p,q}, e^v(x,y) = \left( p^2 x^2 + q^2 y^2 \right)^{-\frac{1}{2}}.$$
Notice that if $\sqrt{\frac{q}{p}} \geq 2$, then by \ref{eqW3}, $\bar{\sigma}_2$ cannot correspond to a second eigenvalue and we obtain a contradiction. We denote $\theta_{\star}$ the minimal value such that for any $\theta:= \frac{q}{p} > \theta_{\star}$, $q$ is not a second Steklov eigenvalue of $E_{p,q}$: we have $\theta_{\star} \leq 4$.

If in addition, $\Phi\left(\mathbb{D}\right)$ is not a flat disk (that is $p<q$), we have an extra mass condition given in \cite{Pet21} (see also \cite{PT23} and Proposition \ref{propcritical}) coming from the choice of the combination $h_{s,t}$ in the variational problem $F_{s,t}$, written as
$$ \frac{p}{q} = \frac{\partial_2 h_{s,t}(p,q)}{\partial_1 h_{s,t}(p,q)} $$
so that $p^{-s} = t q^{-s}$  and 
$$ \bar{\sigma}_1(g_c) = 2\pi t^{-\frac{1}{2s}} \text{ and } \bar{\sigma}_2(g_c) = 2\pi t^{\frac{1}{2s}}$$
It implies that if $s > 0$ then $1 < t \leq \theta_{\star}^s < 4^s$ and that if $s<0$ then $4^s < \theta_{\star}^s \leq t<1$. 

Moreover, we can compare the value of the functional on the flat disk and on $E_{p,q}$
$$ h_{s,t}(2\pi t^{-\frac{1}{2s}}, 2\pi t^{\frac{1}{2s}}) = \frac{\left( 2 \sqrt{t} \right)^{\frac{1}{s}}}{2\pi}  \text{ and } h_{s,t}(2\pi, 2\pi) =  \frac{(1+t)^{\frac{1}{s}}}{2\pi} $$
so that if $s>0$ and $t>1$, then the flat disk $E_{1,1}$ is never a minimizer and if $s<0$, then $E_{p,q}$ for $p<q$ is never a minimizer.

Finally, if $s<0$ and if $\Phi\left(\mathbb{D}\right)$ is a flat disk, we obtain that
$$ \frac{(1+t)^{\frac{1}{s}}}{2\pi} = h_{s,t}(2\pi,2\pi) = \inf_g F_{s,t} < h_{s,t}(0,4\pi) = \frac{t^{\frac{1}{s}}}{4\pi} $$
so that $t < \frac{2^s}{1- 2^s}$.

As a conclusion, we obtain that if 
$$ s>0 \text{ and } t > \theta_{\star}^s \text{ or } s<0 \text{ and } t \geq \frac{1}{2^{-s}- 1} $$
then $\Phi\left( \mathbb{D} \right)$ has to be non planar. 

\medskip

\textbf{(4)}: While we easily prove that planar possibly branched immersed free boundary minimal disks into an ellipsoid by first and second eigenfunctions have to be embedded (see Proposition \ref{propdiffeodim2}), the non-planar case is less obvious. In the current paper, we prove it under symmetry assumptions, 
\begin{theo} \label{theoembedded}
Let $\Phi : \mathbb{D} \to co\left(\mathcal{E}_{\sigma}\right)$ be a non-planar possibly branched free boundary minimal immersion such that $\sigma = (\sigma_1,\sigma_2,\sigma_2)$ only contains first and second Steklov eigenvalues with respect to the critical weight on $\mathbb{S}^1$
$$e^v := \frac{\left\vert \Phi_{\theta} \right\vert}{\left(\sigma_1^2 \phi_0^2 + \sigma_2^2\left(\phi_1^2 + \phi_2^2\right)\right)^{\frac{1}{2}}} = \frac{\left\vert \Phi_{r} \right\vert}{\left(\sigma_1^2 \phi_0^2 + \sigma_2^2\left(\phi_1^2 + \phi_2^2\right)\right)^{\frac{1}{2}}}$$
then $\Phi$ does not have any branched point. Moreover, if
$$ \forall (x,y) \in \mathbb{S}^1, e^{v}(x,y) = e^{v}(-x,y) = e^{v}(x,-y),$$
then $\Phi$ is an embedding.
\end{theo}
Theorem \ref{theoembedded} is a similar result as the embeddedness of (possibly branched) minimal free boundary immersions into balls \textit{by first Steklov eigenfunctions} for surfaces with boundary of \textit{genus $0$} \cite{FS16} but the proof of Theorem \ref{theoembedded} needs much more refinement (see Section \ref{sectionembedding}). We do no know if it is possible to remove the symmetry assumptions on $e^{v}$ and get the same conclusion, or at least the weaker conclusion: $\Phi : \mathbb{S}^1 \to \mathcal{E}_{\sigma} $ parametrizes a simple curve.

At that stage, we do not \textit{a priori} know if a minimizer of $F_{s,t}$ has symmetries. In order to obtain Theorem \ref{theomain}, we then perform a variational method on combinations of eigenvalues under symmetry constraints on the metrics (see Section \ref{sectionoptimizationsymmetries}). Thanks to the symmetry properties of the test functions of Theorem \ref{theotestfunction}, condition \eqref{eqinfcond} is still realized for the infimum among symmetric metrics. Then again, no blow up happens on symmetric minimizing sequences of $F_{s,t}$ for $t \geq t_{\star}(s)$. This leads to the existence of the expected embedded non planar free boundary minimal disk into some ellipsoid. By other simple arguments, we deduce the complete proof of Theorem \ref{theomain} (see Section \ref{sectionotherarguments})

\bigskip

The current paper is fairly self-contained up to the theory of convergence of minimizing Palais-Smale sequences developped in \cite{Pet22}, and all the steps of our construction can also be performed in the analogous closed case (see \cite{Pet23}).
In the end of Section \ref{sectionplanarellipse}, we also suggest ways to make the construction more explicit in the same spirit as the bifurcation methods in the closed case \cite{BP22}. New branches of embedded free boundary minimal disks should appear as soon as the parameter of the ellipsoid crosses values having multiple eigenvalues for a Steklov problem on ellipses endowed with a metric conformal to the flat one. Such a bifurcation method would even give an arbitrary number of embedded disks for sufficiently elongated ellipsoids.

\section{Properties of the free boundary minimal disks into an ellipsoid}
\subsection{Free boundary minimal surfaces into ellipsoids and Steklov eigenvalues}
We start with the following remark: let $\mathcal{E} \subset \mathbb{R}^n$ be an ellipsoid of parameters $\sigma = diag\left(\sigma_1,\cdots,\sigma_n\right)$, with $\sigma_i >0$, defined by
$$ \mathcal{E} = \{ (x_1,\cdots,x_n) \in \mathcal{E} ; \sigma_1 x_1^2 + \cdots + \sigma_n x_n^2 = 1 \} \hskip.1cm,$$
endowed with the induced metric of the Euclidean metric $\xi$. 
We know that $x$ is a harmonic map. We compute the outward normal derivative of $x$ on $\mathcal{E}$:
$$ \partial_{\nu} x = \nu  \hskip.1cm, $$
where the outward normal of the ellipsoid is denoted by $ \nu = \frac{\sigma x}{\left\vert \sigma x \right\vert}$ where $\left\vert \sigma x \right\vert = \left( \sum_{i=1}^n \sigma_i^2 x_i^2\right)^{\frac{1}{2}} $. Therefore, if we endow $\mathcal{E}$ with the metric $g = \frac{\xi}{ \left\vert \sigma x \right\vert  } $, $x_i$ is a Steklov eigenfunction on $(\mathcal{E},g)$ associated to the eigenvalue $\sigma_i$.

\medskip

Now, let $\Phi : (\Sigma,h) \to \mathbb{R}^n$ be such that $\Phi(\partial\Sigma) \in \mathcal{E}$, a $n-1$ dimentional ellipsoid of parameter $\sigma = \left(\sigma_1,\cdots,\sigma_n\right)$. 
A well-known characterisation of $\Phi : (\Sigma,h) \to \mathbb{R}^n$ to be minimal with free boundary in $\mathcal{E}$ is free boundary harmonicity in $\mathcal{E}$ and conformality. We recall that $\Phi$ is harmonic in $\mathcal{E}$ with free boundary if it is a critical point of the energy
$$ E(\Phi) = \frac{1}{2} \int_{\Sigma}\left\vert \nabla \Phi \right\vert_h^2 dA_h $$
under the constraint $\Phi(\partial\Sigma) \subset \mathcal{E}$. The Euler-Lagrange characterization is
$$ \Delta_h \Phi = 0 \hbox{ in } \Sigma \hbox{ and } \partial_{\nu}\Phi \in \left(T_{\Phi}\mathcal{E}\right)^{\perp} \hbox{ on } \partial\Sigma$$
Then $\partial_\nu \Phi = f \nu$ for some function $f = \Phi. \partial_{\nu} \Phi$.
Conformality is characterized by the vanishing of 
$$0 = \left\vert\nabla \Phi \right\vert_h^2 \frac{h}{2} - d\Phi \otimes d\Phi := \sum_{i=1}^n \left(\left\vert\nabla \Phi_i \right\vert_h^2 \frac{h}{2} - d\Phi_i \otimes d\Phi_i \right) \hskip.1cm.$$
For a smooth positive function $e^{2u}$, such that $g = e^{2u} h$, we have
$$ \Delta_g f = e^{-2u} \Delta_h f  \hbox{ and } \partial_{\nu_g} f = e^{-u} \partial_{\nu_h} f \hskip.1cm, $$
and if $\Phi : (\Sigma,h) \to \mathbb{R}^n$ is a minimal isometric immersion with free boundary in $\mathcal{E}$, setting $g = e^{2u} h$ for any function $u$ extending the following formula on the boundary,
$$  e^{u} = \Phi. \partial_{\nu} \Phi = \frac{ 1}{\left\vert \sigma\Phi \right\vert} \hbox{ on } \partial\Sigma \hskip.1cm,$$
the coordinates of $\Phi$ are Steklov eigenfunctions on $(\Sigma,g)$ with eigenvalues $\sigma_1,\cdots,\sigma_n$.

\subsection{The case of planar ellipses} \label{sectionplanarellipse}
Let's consider the simplest example: we assume that $\Phi = id : \mathcal{E}_p \to \mathcal{E}_p$ into a $2$-d ellipse $\mathcal{E}_p$ of parameters $p = (p_1,p_2)$. Setting the metric conformal to $\xi$
$$ g_p = e^{2u} \xi \hbox{ where } e^{u} = \frac{1}{\left\vert \sigma x \right\vert}  =  \frac{1 }{\sqrt{p_1^2 x_1^2 + p_2^2 x_2^2}  } \hbox{ on } \mathcal{E}_p \hskip.1cm,$$
on $(\mathcal{E}_p,g_p)$ we then have 
$$ \Delta_{g_p} x_i = 0 \hbox{ and } \partial_{\nu_p}x_i = p_i x_i $$
for all $i=1,2$. This means that the coordinate functions are eigenfunctions on $(\mathcal{E}_p,g_p)$ with eigenvalues $p_1,p_2$. However, these are not necessarily eigenfunctions associated to the first and second eigenvalue of $(\mathcal{E}_p,g_p)$. Let's compute their renormalized eigenvalue.
By invariance of the indices of eigenvalues by diltation of the ellipse, we study
$\mathcal{E}_p = \{ px^2+y^2 = 1 \}$, for some $0 \leq p \leq 1$, 
where the boundary is endowed with the conformal factor $e^{u_p(x)} = \left(p^2 x^2 + y^2\right)^{-\frac{1}{2}}$. 
We know that there are $k_1 \leq k_2$ such that $\sigma_{k_1}(\mathcal{E}_p,e^{u_p}) = p$ and
$\sigma_{k_2}(\mathcal{E}_p,e^{u_p}) = 1$. 
By parametrisation of the ellipse on the circle $(\frac{1}{\sqrt{p}}\cos\theta,\sin\theta)$, we have that 
$$dL_{\mathcal{E}_p} = \sqrt{\frac{1}{p}\sin^2\theta+\cos^2\theta}d\theta = \frac{1}{\sqrt{p}}e^{-u}d\theta$$
We obtain that the total length of the boundary is
$$ L_{\mathcal{E}_p}(e^{u_p}dL_{\mathcal{E}_p}) = \int_{\mathcal{E}_p} e^{u_p} dL_{\mathcal{E}_p} = \int_{0}^{2\pi} \frac{d\theta}{\sqrt{p}} = \frac{2\pi}{\sqrt{p}} $$
and the renormalized eigenvalues are
$$\bar{\sigma}_{k_1}(\mathcal{E}_p,e^{u_p}) = 
p \sigma_{k_1}(\mathcal{E}_p,e^{u_p}) = p L = 2\pi \sqrt{p} \text{ and } \bar{\sigma}_{k_2}(\mathcal{E}_p,e^{u_p}) = L  = \frac{2\pi}{ \sqrt{p}} $$
Then, for degenerating ellipses, we have that 
$$\bar{\sigma}_{k_1}(\mathcal{E}_p,e^{u_p}) \to 0 \text{ and } \bar{\sigma}_{k_2}(\mathcal{E}_p,e^{u_p}) \to +\infty \hbox{ as } p \to 0 
\hskip.1cm. $$
We know that $\bar{\sigma}_{k}(\mathcal{E}_p,e^{u_p})$ has to be bounded by $2\pi k$ for any $k$. Since $\sigma_{k_2}(\mathcal{E}_p,g_p) = 1$, we have that $k_2 \to + \infty$ as $p \to 0$. It would be interesting to know the value of 
$$ k_{1}(p) = \inf \{ k\in \mathbb{N} ; \sigma_k(\mathcal{E}_p,g_{p}) = p \} \text{ and } k_{2}(p) = \inf \{ k\in \mathbb{N} ; \sigma_k(\mathcal{E}_p,g_{p}) = 1 \} $$
for $i=1,2$ on the Riemannian manifold $(\mathcal{E}_p,g_p)$ defined above. For $p=1$, $k_i(p)= 1$ since the ellipse is a circle. The points $p$ such that $k_i$ jumps are points of bifurcation of eigenvalues. Computing precisely the set $\{ p\in \mathbb{R}^2\setminus\{0\} ; \forall i \in\{1,2\}, k_i(p) \leq 2 \}$ would give the minimal value $\theta_\star$ such that if $\frac{1}{p} > \theta_{\star}$, $k_2(p)\geq 3$. More generally, we conjecture that $k_1(p) = 1$ for any $0\leq p \leq 1$ and that a new bifurcation branch of non planar free boundary minimal surfaces is created if $\frac{1}{p}$ crosses
$$ \theta_\star^k := \inf\left\{ \frac{1}{p} ; 0 \leq p \leq 1 \text{ and } k_2(p) = k \right\} $$
as $p$ decreases.

\subsection{Minimal immersions into ellipsoids do not have branch points at the boundary}

Let $\Phi : \Sigma \to \mathbb{R}^n $ be some (possibly branched) conformal minimal immersion into an ellipsoid
$$ \mathcal{E} = \{ (x_1,\cdots,x_n) \in \mathcal{E} ; \sigma_1 x_1^2 + \cdots + \sigma_n x_n^2 = 1 \} \hskip.1cm,$$
parametrized by $\sigma = diag(\sigma_1,\cdots,\sigma_n)$. We let $ e^{2u} \left(dx^2 + dy^2\right) = \Phi^{\star}eucl$ be the pull-back metric of the Euclidean one by $\Phi$. Notice that branched points correspond to singularities $u(x)=-\infty$. We know that the coordinate functions of $\Phi$ are Steklov eigenfunctions with respect to the conformal facotor $e^{v} = \frac{e^u}{\left\vert \sigma \Phi \right\vert} $ at the boundary. We first prove the following

\begin{cl} \label{clnobranchedpointsboundary}
$\Phi$ cannot have any branched points at the boundary. In other words $e^{u}$ and $e^{v}$ are positive everywhere.
\end{cl}

\begin{proof}
Let $x\in \partial \Sigma$. We set $\psi = \sigma_1 \phi_1(x) \phi_1 + \cdots + \sigma_n \phi_n(x) \phi_n$. We have that for any $y\in \mathbb{D}$,
$$ \psi (y) = \left\langle \sigma \phi(x),\phi(y) \right\rangle \leq \sqrt{\left\langle \sigma \phi(x),\phi(x) \right\rangle}\sqrt{ \left\langle \sigma \phi(y),\phi(y) \right\rangle} = \sqrt{\left\langle \sigma \phi(y),\phi(y) \right\rangle} $$
The function $f: y \mapsto \left\langle \sigma \phi(y),\phi(y) \right\rangle $ is subharmonic since
$$ \Delta f = - \left\langle \sigma \nabla \phi,\nabla\phi \right\rangle \leq 0 \hskip.1cm,$$
so that $f$ realizes its maximum at the boundary, and $f = 1$ at the boundary gives that
$$ \psi(y) \leq 1 = \psi(x) \hskip.1cm.$$
$\psi$ is a harmonic function that realizes its maximum at $x\in \mathbb{S}^1$. Since $\psi$ is harmonic, by the Hopf lemma, we must have $\partial_{\nu} \psi(x)\neq 0$. And we have that
$$ e^{u(x)} = \left\vert \sigma \Phi(x) \right\vert^2 e^{v(x)} = \left\langle \sigma\Phi(x), \partial_{\nu}\Phi(x) \right\rangle = \partial_{\nu} \psi(x) \neq 0 \hskip.1cm. $$

\end{proof}

\subsection{Free boundary minimal disks by first and second eigenfunctions do not have branched points}

The main property of the subsection relies on the following claim:

\begin{cl} \label{clnonvanishinggradatnodalset}
Let $x\in \mathbb{D}$ and let $\psi$ a first or second Steklov eigenfunction on the disk, associated to some positive weight $e^v$ on the boundary. Then 
$$ \psi(x)=0 \Rightarrow \nabla\psi(x) \neq 0 \hskip.1cm. $$
\end{cl}

\begin{proof}
By the Courant nodal theorem, $\psi$ has at most three nodal domains. Moreover, the nodal set is either a smooth curve having two ends on the boundary or the disjoint union of two connected curves having two ends at the boundary. Indeed, since eigenfunctions are non constant and harmonic, the nodal set cannot contain a closed curve, and, the nodal set cannot contain a singularity in the interior of the disk since otherwise, the eigenfunction would have at least four nodal domains. Now let $x\in \mathbb{D}$ be such that $\psi(x)=0$. Let $D$ be a nodal domain such that $x \in \partial D$. Then $\partial D$ is smooth at $x$ and $x$ is an extremal point of the harmonic function $\psi$ on $D$. By the Hopf lemma $ \partial_{\nu}\psi(x) \neq 0$. This ends the proof of the claim.
\end{proof}

It is clear that $\Phi$ has at least two coordinates since first and second eigenfunctions cannot be constant. We know by \cite{jammes} that the multiplicity of first eigenfunctions on the disk is less than 2 and that the multiplicity of second eigenfunctions on the disk is less than 2. Therefore, $\Phi$ has at most 3 coordinates. We consider the cases $n=2$ and $n=3$ separately.

\medskip

\begin{prop}
We first assume $\Phi = (\phi_0,\phi_1,\phi_2) : \mathbb{D}\to \mathbb{R}^3$ is a possibly branched free boundary minimal immersion into $ \mathcal{E} = \{ x\in \R^2 ; \sigma_1 x_0^2 + \sigma_2 (x_1^2+x_2^2)= 1 \}$, where $\sigma_1 < \sigma_2$ and $\phi_0$ is a first eigenfunction and $\phi_1$ and $\phi_2$ are second eigenfunctions. Then $\Phi$ does not have any branched point.
\end{prop}

\begin{proof}
Notice that by Claim \ref{clnobranchedpointsboundary}, $\Phi$ does not have any branched point at the boundary. It remains to prove that for $z \in \mathbb{D}$, we have that $\nabla \Phi (z) \neq 0$. In fact, we will even prove that $\nabla \eta(z) \neq 0$, where $\eta = (\phi_1,\phi_2)$. Let $v \in \mathbb{S}^1$ be such that $\left\langle v, \eta(z) \right\rangle = 0$. Then $\left\langle v, \eta \right\rangle $ vanishes at $z$ and the previous claim implies that
 $ \langle v, \nabla \eta(z) \rangle = \nabla\left(\langle v, \nabla \eta \rangle\right)(z) \neq 0$.
\end{proof}

In the following proposition, we even obtain that planar free boundary minimal disks into ellipsoids by first and second eigenfunction are embeddings.

\begin{prop} \label{propdiffeodim2}
We assume that $\Phi = (\phi_1,\phi_2) : \mathbb{D}\to \mathbb{R}^2$ is a possibly branched free boundary minimal immersion into $ \mathcal{E} = \{ x\in \R^2 ; \sigma_1 x_1^2 + \sigma_2 x_2^2= 1 \}$, where $\sigma_1\leq \sigma_2$ and $\phi_1$ is a first eigenfunction and $\phi_2$ is either a first or second eigenfunction. Then $\Phi$ is a diffeomorphism.
\end{prop}

\begin{proof}
We first prove that the curve $\Phi_{\vert \mathbb{S}^1} : \mathbb{S}^1 \to \mathcal{E}$ is an embedding. 

\medskip

Since $\Phi$ is conformal, we must have that $\left\vert \partial_r \Phi \right\vert^2=\left\vert \partial_{\theta}\Phi \right\vert^2$ on $\mathbb{S}^1$. 
Moreover, by Claim \ref{clnobranchedpointsboundary}, $\left\vert \partial_r \Phi \right\vert^2 = \left\vert \sigma\Phi \right\vert^2 e^{2v}$ never vanishes. 
Therefore, $\Phi_{\vert \mathbb{S}^1} : \mathbb{S}^1 \to \mathcal{E}$ is an immersion. 
Then it is a covering map. Knowing that $\phi_1$ is a first eigenfunction, $\phi_1$ has at most two nodal domains, and the nodal line meets the boundary at two points. 
Therefore, the degree of $\Phi_{\vert \mathbb{S}^1}$ must be equal to $1$ and $\phi_{\vert \mathbb{S}^1}$ is an embedding.

\medskip

As a conclusion, since the ellipse $\mathcal{E}$ encloses a convex curve and $\Phi$ is harmonic in $\mathbb{D}$ a classical result by Kneser \cite{kneser} (see also Choquet \cite{choquet} or a more general result by Alessandrini-Nesi \cite{alessandrininesi}) gives that $\Phi$ is a diffeomorphism.
\end{proof}

In the following section, we also prove embeddedness of immersed non planar free boundary disks into ellipsoids by first and second eigenfunctions with symmetry assumptions.

\subsection{Non planar free boundary minimal disks with symmetry properties are embedded} \label{sectionembedding}

From now to the end of the subsection, we assume that $\Phi = (\phi_0,\phi_1,\phi_2) : \mathbb{D}\to \mathbb{R}^3$ is a free boundary minimal immersion into $ \mathcal{E} = \{ x\in \R^2 ; \sigma_1 x_0^2 + \sigma_2 \left(x_1^2 + x_2^2\right)= 1 \}$, where $\sigma_1 < \sigma_2$ and $\phi_0$ is a first eigenfunction and $\phi_1$ and $\phi_2$ are second eigenfunctions with respect to the positive function $e^v = \Phi. \partial_r \Phi$ on $\mathbb{S}^1$ (positive thanks to Claim \ref{clnobranchedpointsboundary}). We also assume that $e^v$ satisfies the following symmetry assumtions:
$$\forall (x,y)\in \mathbb{S}^1, e^{v(x,y)} = e^{v(-x,y)} = e^{v(x,-y)} $$

\begin{cl} \label{claimsymmetry}
Up to reparametrization and rotation of $\Phi$, we must have that
$$ \forall x,y \in \mathbb{D}, \phi_0(x,-y) = - \phi_0(x,y) \text{ and } \phi_0(-x,y) = \phi_0(x,y) $$
$$ \forall x,y \in \mathbb{D},  \phi_1(x,-y) = \phi_1(x,y)  \text{ and } \phi_1(-x,y) = - \phi_1(x,y) $$
$$ \forall x,y \in \mathbb{D}, \phi_2(x,-y) = \phi_2(x,y) \text{ and } \phi_2(-x,y) = \phi_2(x,y) $$
where $\phi_0$ and $\phi_1$ have exactly $2$ nodal domains and $\phi_2$ has exactly $3$ nodal domains. Moreover, $\phi_2$ does not vanish on $[-1,1]\times\{0\} \cup \{ (0,\pm 1) \}$.
\end{cl}

\begin{proof} The proof the claim is based on the 3 following simple facts

\medskip

\textbf{Fact 1:} For any Steklov eigenfunction $\phi : \mathbb{D}\to \mathbb{R}$, if for all $(x,y)\in \mathbb{D}$, $\phi(x,y) = \phi(-x,y) = \phi(x,-y)$ then $\phi$ has at least $3$ nodal domains and if for all $(x,y)\in \mathbb{D}$, $\phi(x,y) = -\phi(-x,y) = -\phi(x,-y)$ then $\phi$ has at least $4$ nodal domains.

\medskip

\textbf{Fact 2:} If $\phi$ is a Steklov eigenfunction on $\mathbb{D}$ associated to the symmetric weight $e^v : \mathbb{S}^1 \to \mathbb{R}$, then $\phi(-x,y)$ and $\phi(x,-y)$ are also Steklov eigenfunctions.

\medskip

\textbf{Fact 3:} A second eigenfunction $\phi$ vanishes on at most 4 points at the boundary. If $\phi$ vanishes at 2 points $p_1, p_2 \in \partial{D}$, then it has exactly 2 nodal domains and the nodal set is a smooth curve ending at $p_1$ and $p_2$. If $\phi$ vanishes at 3 points $p_0, p_1, p_2 \in \partial \mathbb{D}$, then (up to a permutation of the indices of $p_i$) it has exactly 3 nodal domains and the nodal set is the union of two smooth curves: one ending at $p_0$ and $p_1$, the other one ending at $p_0$ and $p_2$ and they only intersect at $p_0$. If $\phi$ vanishes at 4 points $p_1, p_2, p_3,p_4 \in \partial{D}$, then (up to a permutation of the indices of $p_i$) it has exactly 3 nodal domains and the nodal set is a disjoint union of a smooth curve ending at $p_1$ and $p_2$ and a smooth curve ending at $p_3$ and $p_4$.

\medskip

\textbf{Step 1} Up to a quarter rotation in the set of parametrizations, $\phi_0$ has the expected symmetries, and there are orthogonal second eigenfunctions $\eta_1$ and $\eta_2$ such that
$$\forall (x,y)\in \mathbb{D}, \eta_i(x,y) = \eps_i^1 \eta_i(-x,y) = \eps_i^2 \eta_i(x,-y) .$$ 

\medskip

\textbf{Proof of Step 1}

\medskip

From fact 2, since the first eigenvalue is simple, we must have that 
$$\phi_0(x,y) = \pm \phi_0(-x,y) = \pm \phi_0(x,-y).$$ 
From fact 1 and the Courant nodal theorem ($\phi_0$ have at most two nodal domains), we deduce that, up to a quarter rotation on the set of parametrization, we have the expected symmetries on $\phi_0$.

From fact 2, up to replace $\phi_i$ by $\frac{\phi_i( x,y )+ \phi_i(-x,y) + \phi_i(x,-y) + \phi_i(-x,-y) }{4}$ for $i=1,2$, we can find two Steklov eigenfunctions $\eta_1$ and $\eta_2$ associated to $\sigma_2(e^v)$ such that for $i=1,2$, $j=1,2$, there are $\eps_i^j \in \{\pm 1\}$ such that
$$\forall (x,y)\in \mathbb{D}, \eta_i(x,y) = \eps_i^1 \eta_i(-x,y) = \eps_i^2 \eta_i(x,-y) .$$ 

\medskip

\textbf{Step 2} Up to exchange the indices of $\eta_1$ and $\eta_2$, we have that $\eps_1^2 = \eps_2^2 = 1$, that $\eps_1^1 = -1$ and that $\eps_2^1 = 1$:
$$ \forall x,y \in \mathbb{D},  \eta_1(x,-y) = \eta_1(x,y)  \text{ and } \eta_1(-x,y) = - \eta_1(x,y) $$
$$ \forall x,y \in \mathbb{D}, \eta_2(x,-y) = \eta_2(x,y) \text{ and } \eta_2(-x,y) = \eta_2(x,y) $$

\medskip

\textbf{Proof of Step 2:}

\medskip

Let $i\in \{1,2\}$. From fact 1 and the Courant nodal theorem, we cannot have $ \eps_i^1 = \eps_i^2 = -1$. 

More precisely, let's prove that $ \eps_1^2 = \eps_2^2 = 1$. We assume by contradiction that for some $i\in \{1,2\}$, $\eps_i^2 = -1$. Then $\eps_i^1 = 1$, and knowing that the function $\eta_i$ is orthogonal to $\phi_0$, $\eta_i$ must vanish at some point $(x_0,y_0)\in \mathbb{S}^1$ with $y_0\neq 0$. By symmetries this implies that $\eta_i$ vanishes at $(x_0, - y_0)\in \mathbb{S}^1$. Since $\{y = 0\}$ is a nodal set, there are at least two nodal domains in $\mathbb{D}_+$ and two other nodal domains in $\mathbb{D}_-$. Then $\eta_i$ has at leats four nodal domains: this contradicts the Courant nodal theorem. Then $\eps_1^2 = \eps_2^2 = 1$.

By a similar argument, we can prove that $1 \in \{ \eps_1^1,\eps_2^1 \}$. Indeed, if not, $\eta_1$ and $\eta_2$ vanish on $\{x=0\}$. Since $\eta_1$ is orthogonal to $\eta_2$, symmetries and the Courant nodal theorem give a contradiction.

Now, we set $\eta = (\eta_1,\eta_2)$: we have that $\forall (x,y)\in \mathbb{D}, \eta(x,-y) = \eta(x,y)$. By contradiction, we assume that we also have $\eps_1^1 = \eps_2^1 = 1$, that is $\forall (x,y)\in \mathbb{D}, \eta(-x,y) = \eta(x,y)$. By fact $1$, it follows that for any $v\in \mathbb{S}^1$, $\left\langle \eta,v \right\rangle$ is a second eigenfunction having at least 3 nodal domains and therefore exactly 3 nodal domains by the Courant nodal theorem. By continuity, the map $N : v\in \mathbb{S}^1 \mapsto$\textit{ the number of positive nodal domains of $\left\langle \eta,v \right\rangle$} is a constant map. We obtain a contradiction because for any $v\in \mathbb{S}^1$, $\{ N(v) , N(-v) \} = \{1,2\}$. Then, up to exchange the indices of $\eta_1$ and $\eta_2$, we can assume that $\eps_1^1 = -1$ and that $\eps_2^1 = 1$.

\medskip

\textbf{Step 3:} $\eta_1$ has exactly 2 nodal domains and $\eta_1^{-1}(\{0\}) = \{0\} \times [-1,1]$. $\eta_2$ has exactly $3$ nodal domains and $\eta_2$ cannot vanish on $[-1,1]\times\{0\} \cup \{ (0,\pm 1) \}$.

\medskip

\textbf{Proof of Step 3:}

\medskip

We assume by contradiction that $\eta_2$ vanishes on $\{(\pm 1, 0) , (0,\pm 1) \}$, for instance at $(1,0)$. By symmetry $(x,y) \mapsto (x,-y)$, $(1,0)$ has to be the ending point of at least two smooth vanishing curves. By symmetry $(x,y) \mapsto (-x,y)$, $(-1,0)$ is also a zero of $\eta_2$ and the ending point of at least two smooth curves. This contradicts Fact 3. Then, $\eta_2$ vanishes at $z \in \mathbb{S}^1 \setminus \{ (\pm 1,0),(0,\pm 1) \}$ and by symmetries at $z, \bar{z}, -z, -\bar{z} $. 

By Fact 3,  $\{ z, \bar{z}, -z, - \bar{z} \}$ is 
the nodal set at the boundary and the two disjoint smooth nodal curves 
$C_1$ and $C_2$ must not connect $z$ and $-z$ or $\bar{z}$ and $-\bar{z}$. Let's prove that up to change the indices of $C_i$, $C_1$ 
connects $z$ to $-\bar{z}$ and $C_2$ connects $-z$ to $\bar{z}$. We assume by contradiction that $C_1$ connects $z$ to $\bar{z}$ and $-z$ to $-\bar{z}$. Then, let's study the nodal set of the second eigenfunctions $\left\langle \eta,v \right\rangle$ as $v$ describes $\mathbb{S}^1$. In particular, following Fact 3, we can write
$$ \mathbb{S}^1 = A_2 \sqcup A_3 \sqcup A_4 $$
where $A_k$ is the set of $v\in \mathbb{S}^1$ such that $\left\langle \eta,v \right\rangle$ vanishes at exactly $k$ points at the boundary. $A_2$ and $A_4$ are open sets and $A_3$ is a closed set. For instance, $(0,1)\in A_4$ and $(1,0) \in A_2$. Let $v_0 \in A_3$ be a point at the boundary of the connected component of $A_4$ containing $(0,1)$. Then by symmetries, the point $z_0$ intersecting the two smooth branches of  
the nodal set of $\left\langle \eta,v_0 \right\rangle$ has to belong to $\{(\pm 1,0);(0,\pm 1) \}$. If $z_0 = \pm(0,1)$, then by symmetry of $\left\langle \eta,v_0 \right\rangle$ with respect to $(x,y)\mapsto(x,-y)$, $-z_0$ is also such a singular vanishing point. This contradicts fact 3. Then $z_0 = \pm(1,0)$. However, following the nodal set of $\left\langle \eta,v \right\rangle$ for $v \in ] (0,1),v_0 ] \subset A_4$ we obtain that the nodal set of $\left\langle \eta,v_0 \right\rangle$ contains a smooth curve having two ends at $z_0$ : it does not satisfy fact 3. It is a contradiction.

Therefore, $C_1$ 
connects $z$ to $-\bar{z}$ and $C_2$ connects $-z$ to $\bar{z}$. Finally, for $i=1,2$, $C_i \cap [-1,1]\times \{0\} = \emptyset$ because if we assume $C_i \cap [-1,1]\times \{0\} \neq \emptyset$, then the set contains at least two points and by symmetries, $\eta_2$ would have a nodal domain included in the interior $\mathbb{D}$: this is absurd.

\medskip

\textbf{Conclusion: }

\medskip

Up to replace $\Phi$ by $(\phi_0, \cos\theta \phi_1 + \sin \theta \phi_2, -\sin\theta \phi_1 + \cos\theta \phi_2)$, we can assume that $\phi_1 = \alpha \eta_1$ for some $\alpha \in \mathbb{R}$ and $\phi_1$ have the same symmetries as $\eta_1$. Knowing that
$$ \sigma_1 \phi_0^2 + \sigma_2 \left( \phi_1^2 + \phi_2^2 \right) = 1  \text{ on } \mathbb{S}^1$$
we obtain that $\forall (x,y)\in \mathbb{S}^1, \phi_2(x,y)^2 = \phi_2(x,-y)^2 = \phi_2(-x,y)^2$. By harmonic extension, it is also true in $\mathbb{D}$ and $\phi_2$ must have the same symmetries as $\eta_2$.
\end{proof}

\begin{prop} \label{propembedding}
$\Phi$ is an embedding 
\end{prop}

\begin{proof}
We aim at proving that a projection of the map $\Phi$, the map $\eta:= (\phi_1,\phi_2) : \mathbb{D}_+ \to \eta(\mathbb{D}_+)$ is a diffomorphism. By symmetries, it will prove that $\Phi : \mathbb{D} \to \mathbb{R}^3$ is injective and since it is an immersion, we will obtain that $\Phi$ is an embedding.

\medskip

\textbf{Step 1:} $0 \notin \eta\left( \partial \left({\mathbb{D}_+}\right) \right)$ and $\frac{\eta}{\left\vert \eta \right\vert} : \partial \left({\mathbb{D}_+}\right) \to \mathbb{S}^1$ is a homeomorphism. In other words, $\eta : \partial \left({\mathbb{D}_+}\right) \to \mathbb{R}^2$ is an injective closed curve and up to change the orientation, $\eta\wedge \eta_\theta$ is non-negative on $\mathbb{S}^1_+$ and $ \eta \wedge \eta_x $ is non negative on $[-1,1] \times \{0\}$. In particular, $\eta \left( \partial \left({\mathbb{D}_+}\right) \right)$ encloses a star-shaped domain with respect to $0$.

\medskip

\textbf{Proof of Step 1:} $\phi_1^{-1}(\{0\}) = \{0\}\times (-1,1)$ so that $ \phi_1^{-1}(\{0\}) \cap  \partial \left({\mathbb{D}_+}\right) = \{ (0,0) , (0,1) \} $ and $\phi_2$ does not vanish on this set. Therefore, we can consider $\frac{\eta}{\left\vert \eta \right\vert} : \partial \left({\mathbb{D}_+}\right) \to \mathbb{S}^1$ and prove that it is monotone. In other words, let's prove that for any $v\in \mathbb{S}^1$, and every $x\in \partial \left({\mathbb{D}_+}\right)$ such that $\eta(x) \in D_v $, there is an arc $(a,b)\subset \partial \left({\mathbb{D}_+}\right)$ such that $x \in (a,b)$ and 
\begin{equation*} \forall y \in (a,x), \eta(y) \in H_v^+ \text{ and } \forall y \in (x,b), \eta(y) \in H_v^- \end{equation*}
where 
$$ D_{v} = \{ z \in \R^2 ; \left\langle z, v \right\rangle = 0 \} $$
$$ H_v^{+} = \{ z\in \R^2 ; \left\langle z, v \right\rangle > 0 \}, H_v^{-}  = \{ z\in \R^2 ; \left\langle z, v \right\rangle < 0 \}.$$
We assume that it is not true for some $x \in \partial \left({\mathbb{D}_+}\right)$ and $v\in \mathbb{S}^1$. 

If $x \neq (\pm 1,0)$, $\eta$ is analytic at $x$ and we deduce that
there is an arc $(a,b)\subset\mathbb{S}^1$ such that $x\in (a,b)$ and either for all $y \in (a,b)$, $\eta(y) \in D_v$ or for all $y \in (a,b)\setminus\{x\}$, $\eta(y)\in H_v^{+}$ or for all $y \in (a,b)\setminus\{x\}$, $\eta(y)\in H_v^{-}$. Since $\left\langle\eta,v \right\rangle$ cannot vanish on an arc of $\mathbb{S}^1$, nor on an arc of $[-1,1]\times \{0\}$ by symmetries, up to take $-v$, we assume that for all $y \in (a,b)\setminus\{x\}$, $\eta(y)\in H_v^{+}$. Then, knowing that the nodal set is smooth in the interior of $\mathbb{D}$, $\left\langle \eta,v \right\rangle$ has at least 3 distinct nodal domains at the neighborhood of $x$. By symmetries, $\left\langle \eta,v \right\rangle$ has at least four nodal domains and it is a contradiction.

Now we assume that $x = (\pm 1,0)$. We assume that there is an arc $(a,b)\subset\mathbb{S}^1$ such that $x\in (a,b)$ and either for all $y \in (a,b)\setminus\{x\}$, $\eta(y)\in H_v^{+}$ or for all $y \in (a,b)\setminus\{x\}$, $\eta(y)\in H_v^{-}$. Then, knowing that the nodal set is smooth in the interior of $\mathbb{D}$, $\left\langle \eta,v \right\rangle$ has at least $3$ nodal domains at the neighborhood of $x$ in $\mathbb{D}_+$. By symmetries, $\left\langle \eta,v \right\rangle$ has at least $5$ nodal domains and it is a contradiction. Therefore, such an arc $(a,b)$ does not exist, but by analyticity of $\eta$ on $\partial\left( \mathbb{D}_+\right) \setminus \{ (\pm 1,0) \}$, such an arc $(a,b)$ would exist at the neighborhood of a point $y \in \partial\left( \mathbb{D}_+\right) \setminus \{ (\pm 1,0) \}$ and it is not possible.

In order to conclude step 1, we just notice that $\phi_1$ and $\phi_2$ vanish only twice on $\partial\left( \mathbb{D}_+\right)$, so that the degree of $\frac{\eta}{\left\vert \eta \right\vert} : \partial\left( \mathbb{D}_+\right) \to \mathbb{S}^1$ is $1$ and it is a homeomorphism.

\medskip

\textbf{Step 2:} $\eta_r \wedge \eta_\theta = \eta_x \wedge \eta_y$ does not vanish on $\mathbb{S}^1_+$.

\medskip

\textbf{Proof of Step 2:} We denote $n = e^{-2u} \Phi_x \wedge \Phi_y$ a normal vector of the free-boundary minimal surface. Then $n_0 = e^{-2u}\eta_x \wedge \eta_y$ and $n$ satisfies the equation of harmonic maps into $\mathbb{S}^2$.
Notice that 
$$ \begin{cases} 
n_x = -e e^{-2u} \Phi_x - f e^{-2u}\Phi_y \\
n_y =  -f e^{-2u} \Phi_x + e e^{-2u}\Phi_y
\end{cases}$$
Let $z\in \mathbb{S}^1_+$. Notice in particular that $z\neq (\pm 1,0)$ and $\phi_0(z)\neq 0$. We assume by contradiction that $n_0(z) = 0$. By step 1, $n_0$ has a constant sign at the neighborhood of $z$. Therefore, we must have that $n_{0,\theta}(z) = 0$. Up to a rotation in the set of parametrization, we assume that for any function $f : \mathbb{D}\to \R$, we have $f_r = f_x$ and $f_{\theta} = f_y$ on $\mathbb{S}^1$. By a Taylor expansion,
\begin{equation*}
\begin{split} \Phi(z+h) = & \Phi(z) + h_1 \Phi_x(z)+h_2\Phi_y(z) 
\\ & + \frac{1}{2} \left( (h_1)^2 \Phi_{xx}(z) + 2 h_1 h_2 \Phi_{xy}(z) + (h_2)^2 \Phi_{yy}(z) \right) + o(\left\vert h \right\vert^2 ) 
\end{split} \end{equation*}
Since $n_0(z) = 0$, we have that $\left\langle \nabla \eta(z),\tilde{n}(z) \right\rangle = \left\langle \nabla\Phi(z),n(z) \right\rangle = 0 $ and by the Steklov eigenfunction equation, $\sigma_2 e^{v(z)} \left\langle \eta(z),\tilde{n}(z) \right\rangle =  \left\langle \eta(z),\tilde{n}(z) \right\rangle = 0 $ so that
$$ \left\langle \eta(z+h),\tilde{n}(z) \right\rangle =  \frac{1}{2} \left( \left((h_1)^2 - (h_2)^2\right) e +  2 h_1 h_2 f \right) + o(\left\vert h \right\vert^2 )$$
Since the nodal set of the eigenfunction $\left\langle \eta,\tilde{n}(z) \right\rangle$ has only one branch starting from $z$  (see Step 1), then by the Morse lemma, we must have $e=0$. Since $n_{0,y}(z) = 0$, we obtain that $- f \phi_{0,x}(z) = 0 $. Since by the Steklov eigenfunction equation $\phi_{0,x}(z) = \sigma_1 e^{v(z)} \phi_0(z)$, we obtain that either $\phi_0(z) = 0$, either $f=0$. Since $\phi_0(z) \neq 0$ we obtain $f=0$. Now, since $\left\langle \eta,\tilde{n}(z) \right\rangle$ is harmonic and smooth until the boundary of $\mathbb{D}$, it has a harmonic extension at the neighborhood of $z$. Since $e = f = 0$ at $z$, the nodal set of $\left\langle \eta,\tilde{n}(z) \right\rangle$ has a singularity of multiplicity higher than $3$ at $z$, so that the nodal set $\left\langle \eta,\tilde{n}(z) \right\rangle$ has at least two branches starting from $z$ in $\mathbb{D}$, and it is a contradiction.

\medskip

\textbf{Step 3:} $ \eta \wedge \eta_x $ does not vanish on $(-1,1)\times \{0\}$.

\medskip

\textbf{Proof of Step 3:} We assume that for some $z\in (-1,1)\times \{0\}$, $\eta(z) \wedge \eta_x(z) = 0$. By symmetries, we know that $\phi_{0,x}(z) = 0$ and $\eta_y(z) = 0$. Therefore, if we denote $n = (n_0,\tilde{n})$, and $\varphi = \left\langle \eta, \tilde{n}(z )\right\rangle$, we have obtain that $\nabla\varphi(z) =0$ and $\varphi(z) = 0$ so that the nodal set of $\varphi$ is singular at $z$: it is a contradiction.

\medskip

\textbf{Step 4:} Setting $ e = \left\langle \Phi_{xx},n \right\rangle = - \left\langle \Phi_{yy},n \right\rangle $ and $\tilde{e} = \left\langle \Phi_{xxx},n \right\rangle = - \left\langle \Phi_{yyx},n \right\rangle$, where we denote $n = e^{-2u} \Phi_x \wedge \Phi_y$, then $e$ or $\tilde{e}$ never vanish in $\{ (\pm 1, 0) \}$

\medskip

\textbf{Proof of Step 4:} We have by a Taylor extension that for $z \in \{(\pm1,0)\}$
\begin{equation*}
\eta(z+h) = \sum_{k=0}^{+\infty} \frac{1}{k!} \cos(k\theta) r^k  \eta^{(k)_x}(z)
\end{equation*}
for $h = r (\cos\theta,\sin\theta)$, because $\eta$ is harmonic and satisfies $\eta(x,-y)  = \eta(x,y)$.
Now, thanks to the Steklov equation, $\eta_{x}(z)$ and $\eta(z)$ are parallel vectors and are orthogonal to $\tilde{n}(z)$, where $n(z) = (n_0(z),\tilde{n}(z))$, since $n_0(z) = 0$. Then, if $e=\tilde{e}=0$,
$$ \left\langle  \eta(z+h),\tilde{n}(z) \right\rangle = c \cos(k\theta)r^k  + o(r^k ) $$
for a constant $c\neq 0$ and $k\geq 4$ and it is impossible because of the structure of the nodal set of the second eigenfunction $\left\langle \eta,\tilde{n}(z)\right\rangle$.

\medskip

\textbf{Step 5:} We still denote $\Phi = (\phi_0,\eta) : \mathbb{D}_{r_0} \to \R^2$ a harmonic extension of $\Phi$ where $r_0 >1$. There is $\eps_0>0$ such that for $\eps \leq \eps_0$, there is a smooth function $\eta_\eps : \mathbb{D}_{r_0} \to \mathbb{R}^2$ for $r_0>1$,
satisfying that 
\begin{itemize}
\item $\eta_\eps : \partial\left( \mathbb{D}_+ \right) \to \mathbb{R}^2$ parametrizes a Jordan curve,
\item $\eta_{\eps,x} \wedge \eta_{\eps,y} $ is positive on $\partial\left( \mathbb{D}_+\right)$, 
\item there is a positive smooth function $\sigma_\eps$ such that
$$ -div\left( \sigma_\eps \nabla \eta_{\eps} \right) = 0 \text{ in } \mathbb{D}_+$$
\item $\eta_{\eps}$ converges to $\eta$ in $\mathcal{C}^2\left( \mathbb{D}_{r_0} \right)$ as $\eps \to 0$.
\end{itemize}

\medskip

\textbf{Proof of Step 5:} By Step 1 and Step 2, $\eta_x \wedge \eta_y = \eta_r \wedge \eta_\theta$ is positive on $\mathbb{S}^1_+$. By a straightforword computation, using the Steklov eigenvalue equation
\begin{equation*} 
\begin{split}
\tilde{\eta}_x \wedge \tilde{\eta}_y = \tilde{\eta}_r \wedge \tilde{\eta}_\theta = \frac{1 + \eps \left(1-\frac{\sigma_1}{\sigma_2}\right)\phi_0 }{\left(1+\eps \phi_0\right)^3} \eta_r \wedge \left( \eta_\theta - \frac{\eps \phi_{0,\theta}}{1+\eps \phi_0} \eta \right) 
\\ = \frac{1 + \eps \left(1-\frac{\sigma_1}{\sigma_2}\right)\phi_0 }{\left(1+\eps \phi_0\right)^3} \eta_r \wedge \eta_\theta > 0 
\end{split} \end{equation*}
on $\mathbb{S}^1_+$ for $\eps>0$ small enough. Now, we compute $\tilde{\eta}_x \wedge \tilde{\eta}_y$ on $[-1,1]\times \{0\}$. By the symmetries, we know that $\eta_y = 0$ and $\phi_0 = \phi_{0,x} = 0$. Then
$$ \tilde{\eta}_x \wedge \tilde{\eta}_y = \eta_x \wedge \left(-\eps\phi_{0,y}\eta\right) = \eps \phi_{0,y} \left( \eta \wedge \eta_x \right) $$
on $[-1,1]\times \{0\}$. Up to choose $-\phi_0$, we assume that $\phi_0$ is positve on $\mathbb{D}_+$ so that by the Hopf lemma, $\phi_{0,y}>0$. By step 1 and step 3, $\tilde{\eta}_x \wedge \tilde{\eta}_y$ is then positive on $(-1,1)\times \{0\}$.

Now, we assume that $e(0,1) \neq 0$ and we set $\eta_\eps = \frac{\eta+\alpha_\eps \eta_y}{1+\eps \phi_0}$ where $\alpha_{\eps}>0$ and $\alpha_\eps \to 0$ as $\eps \to 0$. By similar computations as for $\tilde{\eta}$,
$$ \eta_{\eps,x} \wedge \eta_{\eps,y}   \frac{1 + \eps \left(1-\frac{\sigma_1}{\sigma_2}\right)\phi_0 }{\left(1+\eps \phi_0\right)^3} \eta_r \wedge \eta_\theta + \alpha_{\eps} \left( \eta_{xy}\wedge \eta_y + \eta_x \wedge \eta_{yy} \right) + O\left( \alpha_{\eps}^2 + \eps^2 \right) $$
on $\mathbb{S}^1_+$ and
$$ \eta_{\eps,x} \wedge \eta_{\eps,y} = \eta_x \wedge \left(-\eps\phi_{0,y}\eta\right) = \eps \phi_{0,y} \left( \eta \wedge \eta_x \right) + \alpha_{\eps}   \eta_x \wedge \eta_{yy} + O\left( \alpha_{\eps}^2 + \eps^2 \right) $$
on $[-1,1]\times \{0\}$, knowing that $\eta_{xy} = 0$ on $[-1,1]\times \{0\}$. Notice that in both formula, $O\left( \alpha_{\eps}^2 + \eps^2 \right)$ is uniform in $\mathbb{D}$. Moreover, we have by the symmetries $\phi_1(-x,y)=-\phi_1(x,y)$ and $\phi_2(-x,y)=\phi_2(x,y)$ that
$$ \left(\eta_x \wedge \eta_{yy}\right)(-1,0) = \left(\eta_x \wedge \eta_{yy}\right)(1,0)  $$
and that $\eta_x \wedge \eta_{yy}(1,0) = -e(1,0) \eta_x(1,0)\wedge \tilde{n}(1,0) \neq 0$. Then, up to take $-\phi_1$ instead of $\phi_1$ and choosing $\alpha_\eps$ such that $\alpha_{\eps} = o(\eps) $ and $\eps^2 = o(\alpha_\eps)$, then there is $\eps_0>0$ such that for any $\eps \leq \eps_0$, $\eta_{\eps} : \mathbb{D}_{r_0} \to \mathbb{R}^2$ is a local diffeomorphism at the neighbourhood of any point of $\partial{\mathbb{D}_+}$ and $\eta_\eps : \partial\left( \mathbb{D}_+ \right) \to \mathbb{R}^2$ parametrizes a Jordan curve. Finally, setting $\sigma_\eps = (1+\eps \phi_0)^2$, we have
\begin{equation*}
\begin{split}
 -div\left( \sigma_\eps \nabla \eta_\eps \right) =& -div\left( (1+\eps \phi_0)\nabla\left(\eta + \alpha_\eps \eta_y\right) - \left(\eta + \alpha_\eps \eta_y\right) \nabla\left((1+\eps \phi_0)\right)   \right) \\
 =& (1+\eps \phi_0) \Delta\left( \eta + \alpha_\eps \eta_y \right) - \left(\eta + \alpha_\eps \eta_y\right)\Delta(1+\eps \phi_0) = 0
\end{split} 
 \end{equation*}
 and the expected equation is satisfied. Notice that if $e=0$, then by Step 4, $\tilde{e}\neq 0$ and the same argument occurs by setting $\eta_\eps = \frac{\eta+\alpha_\eps \eta_{xy}}{1+\eps \phi_0}$.

\medskip

\textbf{Conclusion:} 
By Alessandrini and Nesi \cite{alessandrininesi2}, we obtain that $\eta_{\eps} : \Omega_{\eps} \to \eta_{\eps}(\Omega_{\eps})$ is a diffeomorphism, where $\Omega_{\eps}$ is chosen as a smooth domain such that $\mathbb{D}_+ \subset \Omega_{\eps} \subset \mathbb{D}_{r_0}$ and $\eta_{\eps}: \partial\Omega_\eps \to \R^2$ is a parametrization of a Jordan curve. Then $\eta_{\eps} : \mathbb{D}_+ \to \eta(\mathbb{D}_+)$ is a diffeomorphism. Letting $\eps \to 0$, we obtain that $\eta_{x}\wedge \eta_y \geq 0$ on $\mathbb{D}_+$. The zero set of $\eta_x \wedge \eta_y$ is the same as the zero set of the first coordinate $n_0$ of the normal vector to the surface, which satisfies an eigenfunction equation $ \Delta n_0 = \left\vert \nabla n \right\vert^2 n_0 $ (because $n$ is a harmonic map into the sphere). $n_0$ has only one nodal domain. Then $\eta_{x}\wedge \eta_y >0$ on $\mathbb{D}_+$. 
Then $\eta_{\eps}^{-1}$ converges in $\mathcal{C}^2(K)$ for any compact subset $K$ of $\eta\left(\mathbb{D}_+\right)$ to a function $\psi : \eta\left(\mathbb{D}_+\right) \to \mathbb{D}_+$ that is the inverse function of $\eta$. In particular, $\eta$ is injective on $\mathbb{D}_+$. By symmetries, $\Phi$ is injective on $\mathbb{D}$ and we obtain the expected result.
\end{proof}

\section{Optimization of combinations of Steklov eigenvalues} \label{sectionoptimizationsymmetries}

In the current section, we prove Theorem \ref{theomain}. In particular, we gather the previous techniques for maximization of combinations of Steklov eigenvalues \cite{Pet21} and \cite{Pet22}, and adapt them for the first time on a restriction of the admissible conformal factors at the boundary $e^{v}$ to symmetric ones on a surface $\Sigma$. For simplicity, we choose to write it in our case $\Sigma = \mathbb{D}$ and
$$ \forall x,y \in \mathbb{S}^1, e^{v(x,y)} = e^{v(-x,y)} = e^{v(x,-y)} $$
but it appears that such a method for symmetry constraints can be setted in a general framework.

\subsection{Description of critical metrics for combinations of first and second Steklov eiegenvalues}

We denote $\sigma_k(f)$ is the $k$-th non zero Steklov eigenvalue with respect to the density $f:\mathbb{S}^1 \to \mathbb{R}^*_+$ on $\mathbb{D}$, that is the restriction at the boundary of the density of the measure associated to some metric conformal to the flat one on the disk. We set $\bar{\sigma}_k(f) = \sigma_k(f) L_f(\mathbb{S}^1)$, where $L_f\left(\mathbb{S}^1\right) = \int_{\mathbb{S}^1} f d\theta$ is the length of the boundary with respect to $f$.  We set $ V = \mathcal{C}^0\left( \mathbb{S}^1\right)$ and
$$ X = \{f \in \mathcal{C}^0( \mathbb{S}^1), f>0 \text{ and } \forall x,y \in \mathbb{S}^1, f(x,y) = f(-x,y) = f(x,-y) \} $$ 
and for $f\in X$ we set
$$ E(f):=  F\left(\bar\sigma_1(f),\cdots, \bar\sigma_m(f) \right)$$
where $F : \mathbb{R}^m \to \mathbb{R}_+$ is a $\mathcal{C}^1$ function such that for all $i\in \{1,\cdots,m\}$, $\partial_i F \leq 0$ everywhere. Notice that in the current paper, $F = h_{s,t}$ only depends on first and second Steklov eigenfunctions.  

We denote $\partial E(f)$ the Clarke subdifferential of $E$ at $f$ (see \cite{PT23}). We proved in \cite{PT23}

\begin{equation} \label{eqsubsteklov}
\begin{split}
 \partial E(f) \subset \overline{co}_{(\phi_1,\cdots,\phi_m) \in \mathcal{O}_m} \left\{   \sum_{i=1}^m d_i \bar\sigma_i(f) \left( \frac{1}{\int_{\mathbb{S}^1} f} - \left(\phi_i\right)^2 \right) \right\} 
 \end{split},
\end{equation}
where $(\phi_1,\cdots,\phi_m)$ lies in the set $\mathcal{O}_{m}(f)$ of $L^2(\partial\Sigma, f d\theta)$-orthonormal families where $\phi_i \in E_i(f)$ ($E_i(f)$ is the set of $i$-th eigenfunctions) and $d_i =  \partial_i F\left(\bar\sigma_1(f),\cdots,\bar\sigma_m(f) \right)  \leq 0$. Notice that in this case the subdifferential is a space of functions on $\mathbb{S}^1$, while it is defined as a subspace of $V^\star $, the set of Radon measures on $\mathbb{S}^1$: here, we identify a function $\psi$ to the measure $\psi d\theta$ on $ \mathbb{S}^1$.

From this (see also \cite{Pet21} or Proposition \ref{prop:constructPSK} for the equivariant context in the next subsection with $\theta_\eps = 0$) and the multiplicity results by \cite{jammes} for the first and second eigenvalues of the disk (and Proposition \ref{propdiffeodim2} and Proposition \ref{propembedding} for the embedding part), we obtain in our case $F=h_{s,t}$ that the critical densities $e^u$ (densities such that $0 \in  \partial E(e^u)$) satisfy:
\begin{prop} \label{propcritical} If $e^u$ is a critical density of $h_{s,t}(\bar{\sigma}_1,\bar{\sigma}_2) = \left( \bar{\sigma}_1^{-s} + t \bar{\sigma}_2^{-s} \right)^{\frac{1}{s}}$, then there is a free boundary minimal immersion $\Phi =(\phi_0,\phi_1,\phi_2) : \mathbb{D} \to co(\mathcal{E}_\sigma)$, where $\sigma = (\bar{\sigma}_1(e^u),\bar{\sigma}_2(e^u),\bar{\sigma}_2(e^u))$, $\phi_0$ is a Steklov eigenfunction associated to $\sigma_1(e^u)$ and $\phi_1$ and $\phi_2$ are associated to $\sigma_2(e^u)$ such that $(\phi_0,\phi_1,\phi_2)$ are independent function in $L^2(\mathbb{S}^1,e^ud\theta)$ or $\phi_2 = 0$ and $(\phi_0,\phi_1)$ are independent. Moreover,
$$ \int_{\mathbb{S}^1} \phi_0^2 e^u d\theta = \frac{ \bar{\sigma}_1(e^u)^{-s-1} }{f_{s,t}(\bar{\sigma}_1(e^u),\bar{\sigma}_2(e^u))} \int_{\mathbb{S}^1} e^u d\theta $$
and
 $$ \int_{\mathbb{S}^1} \left(\phi_1^2+ \phi_2^2\right) e^u d\theta = \frac{ t  \bar{\sigma}_2(e^u)^{-s-1} }{f_{s,t}(\bar{\sigma}_1(e^u),\bar{\sigma}_2(e^u))} \int_{\mathbb{S}^1} e^u d\theta. $$
If in addition $f$ is symmetric with respect to the two axes of $\mathbb{S}^1$, then $\Phi$ is an embedding.
\end{prop}

\subsection{Construction and convergence of almost critical equivariant minimizing sequences}

The current subsection describes in our equivariant setting the construction of an almost critical (or Palais-Smale-like) minimizing sequence with the methods in \cite{Pet22} in order to obtain Theorem \ref{theo:MWsconvergence}. We define for a function $F : \R^2 \to \R_+$ such that $\partial_i F \leq 0$
$$ I_F^{eq}(\mathbb{D}) := \inf_{e^u \in X^{eq}(\mathbb{D})} F( \bar{\sigma}_1(e^u), \bar{\sigma}_2(e^u) ) $$
where $eq$ stands for "equivariant" and denoting $\mathcal{C}^0_{>0}(S)$ the set of positive continuous functions on a curve $S$,
$$ X^{eq}(\mathbb{D}) = \{f \in \mathcal{C}^0_{>0}( \mathbb{S}^1),  \forall (x,y) \in \mathbb{S}^1, f(x,y) = f(-x,y) = f(x,-y) \} $$ 
and
$$ I_F^{eq}(\mathbb{D}\sqcup \mathbb{D}) := \inf_{e^u \in X^{eq}\left(\mathbb{D}\sqcup \mathbb{D}\right)} F( \bar{\sigma}_1(e^u), \bar{\sigma}_2(e^u) )  $$
where denoting $\Sigma := \left(\mathbb{D}\right)_1 \sqcup \left(\mathbb{D}\right)_2$ the disjoint union of two copies of flat disks and its boundary $\partial \Sigma = \left(\mathbb{S}^1\right)_1 \sqcup \left(\mathbb{S}^1\right)_2$,
$$X^{eq}\left(\Sigma\right) := \left\{ f \in \mathcal{C}^0_{>0}\left( \partial \Sigma\right), \begin{cases}f_{\left(\mathbb{S}^1\right)_1} = f_{\left(\mathbb{S}^1\right)_2} \\ \forall (x,y) \in \left(\mathbb{S}^1\right)_i, f(x,y) = f(-x,y)\end{cases}  \right\}  $$

\begin{theo} \label{theo:MWsconvergence}
For any minimizing sequence $e^{u_n} d\theta$ for $I_F^{eq}(\mathbb{D})$, we have up to the extraction of a subsequence that $\left(e^{u_n} d\theta \right)$ MW$\star$-bubble-tree-converges to a minimizer of $I_F^{eq}(\mathbb{D} \sqcup \mathbb{D})$.
The conformal factors of the minimizer are positive and smooth.
Moreover, if 
$$ I_F^{eq}(\mathbb{D} \sqcup \mathbb{D}) < I_F^{eq}(\mathbb{D}) $$ 
then up to the extraction of a subsequence, any minimizing sequence MW$\star$-converges to a minimizer of to a positive and smooth conformal factor on $\Sigma$.
\end{theo}
Here MW$\star$ convergence denotes the weak-star convergence with respect to the dual space of Radon measures. The definition of MW$\star$ bubble-tree-convergence is given in \cite{Pet22}, Definition 2.1.

As in \cite{Pet22}, we will work with generalized eigenvalues defined on $\overline{X}^{eq}(\mathbb{D})$ which is the completion of 
$\{ (\varphi,\psi)\mapsto \int_{\mathbb{S}^1} \varphi \psi e^{u} d\theta ; e^u \in X^{eq}(\mathbb{D})  \}$ 
with respect to the following norm on continuous non negative bilinear forms $\beta : H^1(\mathbb{D})\times H^1(\mathbb{D}) \to \R_+$
$$\Vert \beta \Vert_{g}  := \sup_{(\varphi,\psi)\in H^1(\mathbb{D})\times H^1(\mathbb{D})} \frac{\vert \beta(\varphi,\psi) \vert}{\Vert \varphi \Vert_{H^1(g)}\Vert \psi \Vert_{H^1\left(g\right)}}$$
where for a metric $g$ on $\mathbb{D}$,
$$ \Vert \varphi \Vert_{H^1(g)}^2 := \int_{\mathbb{D}} \vert \nabla \varphi \vert^2_g dA_g + \int_{\mathbb{S}^1} \varphi^2 dL_g.$$
We proved in \cite{Pet22} that generalized eigenvalues are Lipschitz, on $\overline{X}^{eq}(\mathbb{D})$ (see Proposition 1.1), that we have existence of eigenfunctions (consequences of Proposition 1.3) and that we can compute their right derivative at $\beta \in \overline{X}^{eq}(\mathbb{D})$ with respect to variations  $b \in \overline{X}^{eq}(\mathbb{D})$ ($\beta+ t b$ as $t \searrow 0$). Indeed, denoting by $E_k(\beta)$ the set of $k$-th generalized Steklov eigenfunctions with respect to $\beta$ and
$$ i(\beta,k) := \min\{ i \in \mathbb{N}^* ; \sigma_i(\beta) = \sigma_k(\beta) \}$$
$$ I(\beta,k) := \max\{ i \in \mathbb{N}^* ; \sigma_i(\beta) = \sigma_k(\beta) \}$$
we have

\begin{prop} \label{prop:firstderivative}
For $\beta \in  \bar{X}$, and $b \in  \bar{X}$,
\begin{equation} \label{eq:firstderivativeminmax}
\begin{split} \lim_{t \searrow 0 } \frac{\bar{\sigma}_k(\beta+tb) - \bar{\lambda}_k(\beta)}{t} = & \bar{\sigma}_k(\beta) \left( b(1,1) - \min_{ V \in \mathcal{G}_{k-i(k)+1}(E_k(\beta)) } \max_{\phi \in V\setminus \{0\}} \frac{b(\phi,\phi)}{\beta(\phi,\phi)} \right) \\ 
= & \bar{\sigma}_k(\beta) \left( b(1,1) - \max_{ V \in \mathcal{G}_{I(k)-k+1}(E_k(\beta)) } \min_{\phi \in V\setminus \{0\}}   \frac{b(\phi,\phi)}{\beta(\phi,\phi)} \right)
\end{split}
\end{equation}
In addition there is an orthonormal basis $(\phi_i,\cdots,\phi_{I})$ with respect to $\beta$, which is orthogonal with respect to $b$ such that for any $i \leq j \leq I$, 
\begin{equation}\label{eq:eigenbasissymm}\forall (x,y)\in \mathbb{D}, \left(\phi_j(x,y)\right)^2 = \left(\phi_j(-x,y)\right)^2= \left(\phi_j(x,-y)\right)^2 \end{equation}
and  
$$ \lim_{t \searrow 0 } \frac{\bar{\sigma}_k(\beta+tb) - \bar{\lambda}_k(\beta)}{t} = \bar{\sigma}_k(\beta+tb) \left( b(1,1) - b(\phi_k,\phi_k) \right) $$
\end{prop}

The proof of \eqref{eq:firstderivativeminmax} in Proposition \ref{prop:firstderivative} is given in \cite{Pet22}, Proposition 1.4. The second part of the proposition gives an extra property in the equivariant version. The min-max is a finite dimensional characterization of eigenvalues of the endomorphism $L$ endowed by $b$ with respect to the scalar product $\beta$. If $\phi \in E_k(\beta)$ is an eigenfunction of $L$, then $(x,y)\mapsto \phi(x,-y)$ and $(x,y)\mapsto \phi(-x,y)$ are also eigenfunctions of $L$ since $\beta$ and $b$ are equivariant bilinear forms. Then by symmetrizations, we can construct an orthonormal basis of eigenfunctions of $L$ such that \eqref{eq:eigenbasissymm} holds.

We then deduce the following proposition
\begin{prop} \label{prop:constructPSK}
For any $\eps>0$, we let $e^{u_\eps}$ be a conformal factor in $X^{eq}(\mathbb{D})$, and $g_\eps:= e^{2u_\eps} g_{\mathbb{D}}$ where $e^{2u_\eps}$ is a positive continuous extension to $\mathbb{D}$ of $e^{2u_\eps}$ such that
$$ E(g_\eps,e^{u_\eps}d\theta) \leq \inf_{\beta\in \overline{X^{eq}(\mathbb{D})}} E(\beta) + \eps^2. $$ 
Then, there is a $(PS)^{eq}_2$ sequence $(\beta_\eps,\Phi_\eps,g_\eps)$ as $\eps\to 0$.
\end{prop}
Here, we define the $(PS)_2^{eq}$ sequences as in \cite{Pet22} $(K=2)$
\begin{defi} \label{def:PS} Let $\beta_\eps \in \overline{X}^{eq}(\mathbb{D})$, $\Phi_{\eps} : \mathbb{D} \to \mathbb{R}^{m_\eps}$ be a sequence of maps with $(m_\eps)_{\eps>0} \in \left(\mathbb{N}^*\right)^{\R_+^*}$, $g_\eps := e^{2u_\eps} g_{\mathbb{D}}$ a family of metrics conformal to $g_{\mathbb{D}}$ such that $e^{u_\eps} \in X^{eq}(\mathbb{D})$ and $K \in \mathbb{N}^\star$. We say that $(\beta_\eps,\Phi_\eps,g_\eps)$ satisfies the equivariant Palais-Smale assumption (with eigenvalue indices bounded by $K$) $(PS)^{eq}_{K}$ as $\eps\to 0$, if
\begin{itemize}
\item The diagonal terms of $\sigma_{\eps} := diag(\sigma_{1}^{\eps},\cdots, \sigma_{m_{\eps}}^{\eps})$ are the $m_\eps$ first Steklov eigenvalues associated to $\beta_\eps$ such that $\sigma_{1}^{\eps}\leq \cdots \leq \sigma_{m_{\eps}}^{\eps} = \sigma_K^\eps$ where $\sigma_K^\eps$ is the $K$-th eigenvalue.
\item $ \Delta_g \Phi_{\eps} = \beta_\eps\left( \sigma_{\eps} \Phi_{\eps} , .\right)$, where $\beta_\eps\left( \sigma_{\eps} \Phi_{\eps} , . \right) : H^1(\mathbb{D})^{m_\eps} \to \R$
\item $L_\eps(1) = L_\eps\left( \left\vert \Phi_{\eps} \right\vert_{\sigma_{\eps}}^2 \right) = \int_{\mathbb{D}}\left\vert \nabla\Phi_{\eps} \right\vert^2 dA_{\mathbb{D}} =  1$ where we denote $L_\eps$ the linear form associated to $\beta_\eps$
\item For $i\in \{1,\cdots,m_\eps\}$, $t_i^{\eps} = L_\eps \left( \left(\phi_i^\eps\right)^2 \right)$ and $\sum_{i=1}^{m_\eps} \sigma_{i}^\eps t_i^\eps  = 1$ and we have that $\sigma_i^\eps t_i^\eps \to 0$ as $\eps \to 0$ for any $i \in \{1,\cdots,m_\eps\}$ such that $\sigma_i^\eps \to 0$ as $\eps \to 0$.
\item $\left\vert \Phi_\eps \right\vert_{\sigma_{\eps}}^2 \geq_{a.e} 1 - \theta_\eps^2$ in $\partial\Sigma$ where $\Vert \theta_{\eps} \Vert_{H^1(g_\eps)}^2 \leq \eps$
\item $\Vert \beta_\eps - e^{u_\eps}dL_g \Vert_{g_\eps} \leq \eps$
\end{itemize}
\end{defi}

The proof of Proposition \ref{prop:constructPSK} is similar to the proof of Proposition 1.5 in \cite{Pet22}. The key to rewrite the proof in the equivariant context is to use the second part of Proposition \ref{prop:firstderivative}: theanks to that property, the separation argument can be writen in space of $L^2$-functions $f$ such that 
$$\forall (x,y)\in \mathbb{D}, f(x,y)=f(-x,y) = f(x,-y).$$

From Proposition 3.1 in \cite{Pet22} and Proposition \ref{prop:constructPSK}, we then obtain Theorem \ref{theo:MWsconvergence}.

\subsection{Proof of Theorem \ref{theomain}} \label{sectionotherarguments}

Thanks to the previous subsection, we have existence of a maybe disconnected minimizer for $g\mapsto F_{s,t}(g)$ under symmetry constraints. 
Thanks to Theorem \ref{theotestfunction}, involving symmetric test functions, there is in fact at most one connected component at the limit by comparison of the energy between one disk and two disjoint disks. By symmetry assumptions, if there is one concentration point, we must have that $e^{u_\eps} d\theta $ weak $\star$ converges to a sum $\frac{\delta_1 + \delta_{-1}}{2}$ or $\frac{\delta_i + \delta_{-i}}{2}$ which is impossible. Then the MW$\star$-bubble-tree-convergence of Theorem \ref{theo:MWsconvergence} is in fact a MW-$\star$ convergence. We obtain the existence of smooth minimizers on the disk for the symmetric problem. By Theorem \ref{theoembedded}, we obtain embedded free boundary minimal disks into ellipsoids. The remaining arguments to obtain non planar ones are in the introduction. Let's verify the remaining properties of Theorem \ref{theomain}

\subsubsection{Monotonicity} We prove that $L_{s,t}p_{s,t}$ is decreasing and that $L_{s,t}$ is increasing. Let $t_1 < t_2$, we set for $i=1,2$: $a_i = L_{s,t_i}p_{s,t_i}$ and $b_i = L_{s,t_i}$ the first and second Steklov eigenvalues of the spectral problem. We have
$$ h_{s,t_1}(a_1,b_1) \leq h_{s,t_1}(a_2,b_2) 
\text{ and } h_{s,t_2}(a_2,b_2) \leq h_{s,t_2}(a_1,b_1) $$
so that if $s<0$,
$$ a_1^{-s} + t_1 b_1^{-s} \geq a_2^{-s} + t_1 b_2^{-s} \text{ and } a_2^{-s} + t_2 b_2^{-s} \geq a_1^{-s} + t_2 b_1^{-s} $$
and by sum,
$$  (t_2-t_1) (b_2^{-s} -b_1^{-s}) \geq 0 $$
and dividing the equations by $t_i$, and a sum,
$$ \left(\frac{1}{t_1} - \frac{1}{t_2}\right) \left( a_1^{-s} - a_2^{-s} \right) \geq 0 $$
We argue the same way if $s>0$.

\subsubsection{Energy estimates}
By Theorem \ref{theotestfunction}, we have that for any $s\neq 0$ $t>0$ and $\eps >0$, $ h_{s,t}(D_{s,t},g_{s,t}) \leq h_{s,t}(\mathbb{D},h_\eps) $ and a straightforward computation gives on $\sigma_t := \bar{\sigma}_1\left(D_{s,t},g_{s,t}\right) $ that
$$ \sigma_t \geq \left( a_\eps^{-s} + t \left( b_\eps^{-s} - \bar{\sigma}_2\left(D_{s,t},g_{s,t}\right)^{-s} \right)  \right)^{-\frac{1}{s}} \geq \left( a_\eps^{-s} + t \left( b_\eps^{-s} - \left(4\pi\right)^{-s} \right)  \right)^{-\frac{1}{s}} , $$
where $a_\eps = \frac{2\pi}{\ln\left( \frac{1}{\eps}\right)} + O \left(  \frac{1}{\ln\left( \frac{1}{\eps}\right)^2} \right)$ and $b_\eps = 4\pi - 16\pi \eps +o(\eps)$ as $\eps \to 0$, so that
$$ \frac{\sigma_t}{2\pi} \geq \ln\left(\frac{1}{\eps}\right)^{-1} \left( 1 + O\left( \frac{1}{\ln\left(\frac{1}{\eps}\right)} \right) +  t \eps \ln\left(\frac{1}{\eps}\right)^{-s} s (4\pi)^{-s} 16\pi (1 + o(1)) \right)^{-\frac{1}{s}}  $$
as $\eps\to 0$, so that choosing $ t = \eps^{-1} \ln \left(\frac{1}{\eps}\right)^{s-1}$  we have
\begin{equation} \label{eqintx0} \bar{\sigma}_1\left(D_{s,t},g_{s,t} \right) \geq \frac{2\pi}{\ln t}\left(1 + O\left(\frac{1}{\ln t}\right)\right) \end{equation}
as $t\to +\infty$. In addition, we have that for any $s\neq 0$
\begin{equation*} \bar{\sigma}_2\left(D_{s,t},g_{s,t}\right) 
 \geq \begin{cases} \left( \frac{- \bar{\sigma}_1\left(D_{s,t},g_{s,t}\right)^{-s}}{t}  +   (4\pi)^{-s} 
 \right)^{-\frac{1}{s}} \text{ if } s<0 \\ \left( \frac{a_\eps^{-s} }{t}  +   b_\eps^{-s}  
 \right)^{-\frac{1}{s}} \text{ if } s>0  \end{cases}, \end{equation*}
so that if $s<0$, we obtain $\bar{\sigma}_2\left(D_{s,t},g_{s,t}\right) \geq 4\pi - O\left(\frac{1}{t}\right) $ as $t\to +\infty$ and if $s>0$, 
$$ \bar{\sigma}_2\left(D_{s,t},g_{s,t}\right) \geq 4\pi \left( \frac{\ln\left(\frac{1}{\eps}\right)^s 2^s}{t}\left(1+O\left(\frac{1}{\ln\left(\frac{1}{\eps}\right)}\right)\right) + 1 + 4s \eps (1+o(1)) \right)^{-\frac{1}{s}} $$
as $\eps \to 0$, so that setting $ t = \eps^{-1} $,
\begin{equation} \label{eqintx02} \bar{\sigma}_2\left(D_{s,t},g_{s,t}\right) \geq 4\pi - O\left( \frac{\left(\ln t\right)^{\max\{s,0\}}}{t} \right) \end{equation}
as $t\to +\infty$.

Knowing that the critical metrics of $F_{s,t}$ satisfy the mass conditions of Proposition \ref{propcritical} coming from the choice of the combination $h_{s,t}$, we obtain
$$ t = \frac{\int_{\partial D_{s,t}} (x_1^2 + x_2^2)dL_{g_{s,t}} }{\int_{\partial D_{s,t}} x_0^2 dL_{g_{s,t}}} \frac{1}{p_{s,t}^{s+1}} 
$$
for any $s$, and since $L_{s,t} = p_{s,t}\int_{\partial D_t} x_0^2 dL_{g_{s,t}}+ \int_{\partial D_{s,t}} (x_1^2 + x_2^2)dL_{g_t}$, we obtain
$$\int_{\partial D_{s,t}} x_0^2 dL_{g_{s,t}} = \frac{L_{s,t}}{t p_{s,t}^{s+1} + p_{s,t}} \to 0 \text{ as } t \to +\infty,$$
since $\bar{\sigma}_1\left(D_{s,t},g_{s,t} \right) = p_{s,t}L_{s,t} $ satisfies \eqref{eqintx0} and $\bar{\sigma}_2\left(D_{s,t},g_{s,t} \right) = L_{s,t} \to 4 \pi$ by \eqref{eqintx02}. Then $x_0$ converges to $0$ in $H^1\left(D_{s,t}\right)$.

\subsubsection{Varifold convergence to disk with multiplicity 2}
From the very last of the subsection 2.2.2, we deduce that up to take a conformal diffeomorphism $\Phi_{s,t} : \mathbb{D} \to D_{s,t}$ of coordinates $(\phi_0^{s,t}, \phi_1^{s,t} ,\phi_2^{s,t})$,  $(\phi_1^{s,t},\phi_2^{s,t})$ with the sequence of induced metrics $\Phi_{s,t}^\star eucl = e^{2u_{s,t}}g_{\mathbb{D}}$ is a $(PS)_2$ sequence as $t\to +\infty$. We use again Proposition 3.1 in \cite{Pet22}, to obtain the MW$\star$ bubble tree convergence of $e^{u_{s,t}}d\theta$. We must have $2$ disks at the limit since $\bar{\sigma}_2(e^{u_{s,t}}d\theta) \to 4\pi$ as $t\to +\infty$.

In addition, following the proof of Proposition 3.1, and in particular Claim 3.7, we obtain that $(\phi_1^{s,t},\phi_2^{s,t})$ $H^1$-bubble tree converges as $t\to +\infty$. This implies the varifold convergence of $D_{s,t}$. This completes the proof of Theorem \ref{theomain}.

\section{Spectral gaps} \label{spectralgaps}

\subsection{Notations and preliminaries}

On the Euclidean disk $(\mathbb{D},\xi)$, we set the following competitor $g_{\eps} = e^{2\omega_{\eps}}\xi$, where for $z = (z_1,z_2) \in \mathbb{D}^2$ we set
$$ e^{\omega_{\eps}(z)} = \frac{\beta_{\eps}^2 - 1}{\left\vert \beta_{\eps} - z\right\vert^2 } + \frac{\beta_{\eps}^2 - 1}{\left\vert \beta_{\eps} + z\right\vert^2 } \hskip.1cm,$$
for
$$ \beta_{\eps} = \frac{1+\eps}{1-\eps} > 1 \hskip.1cm.$$
Let's describe the surface $(\mathbb{D},g_{\eps})$ in conformal coordinates. 
We often prefer to be in the chart of the $Im(x)\geq 0$ half plane $\mathbb{H}_+$ via the holomorphic maps $f_{\pm} :\mathbb{D}\to \mathbb{H}_+ $ defined by
$$ f_{\pm}(z) = i \frac{1 \mp z}{1 \pm z} \hbox{ and } f_{\pm}^{-1}(x) = \pm \frac{1+ i x}{1- ix} \hskip.1cm. $$
Notice that 
$$ \left( f_{\pm}^{-1} \right)^{\star} \xi = \frac{4}{ \left\vert 1-ix \right\vert^4 } \xi = \frac{4}{\left( \left(1+ x_2 \right)^2+ x_1^2\right)^2}  \xi = e^{2u}\xi $$
where the function
$$ u(x_1,x_2)=\ln\left( \frac{2}{\left(1+x_2\right)^2 + x_1^2}\right) $$
satisfies the Liouville equation
$$ \begin{cases} \Delta u = 0 \text{ in } \mathbb{H}_+ \\
-\partial_{y} u = e^u \text{ on } \mathbb{R}\times\{0\}\hskip.1cm.
\end{cases} $$
Notice that the first Steklov eigenfunctions on the flat disk $z_1$ and $z_2$ become on $\mathbb{H}_+$ endowed with the metric $e^{2u}\xi$:
\begin{equation}
\begin{split}
Z_1(x) = Re(f_+^{-1}(x)) = \frac{1-\left\vert x \right\vert^2}{(1+x_2)^2+(x_1)^2} \\  Z_2(x) = Im(f_+^{-1}(x)) = \frac{2 x_1}{(1+x_2)^2+(x_1)^2} \hskip.1cm.
\end{split}
\end{equation}

\medskip

The function $\omega_{\eps} : \mathbb{D}\to \mathbb{R} $ was defined such that $\omega_{\eps}:= \tilde{\omega}_{\eps}\circ f_+ - u \circ f_+$ where
$$ \tilde{\omega}_{\eps} = \ln\left( e^{\tilde{u}_{\eps}} + e^{\tilde{v}_{\eps}} \right), \hspace{3mm} e^{\tilde{u}_{\eps}(x)}= \frac{1}{\eps}e^{u\left(\frac{x}{\eps}\right)} \text{ and } e^{\tilde{u}_{\eps}(x)}= \eps e^{u\left(\eps x\right)}$$
where we notice that $u_{\eps}:= \tilde{u}_{\eps}\circ f_+ - u \circ f_+$ and $v_{\eps}:= \tilde{v}_{\eps}\circ f_+ - u \circ f_+$ satisfy
$$ e^{2u_{\eps}} = \frac{\beta_{\eps}^2 - 1}{\left\vert \beta_{\eps} - z \right\vert^2} \text{ and } e^{2v_{\eps}} = \frac{\beta_{\eps}^2 - 1}{\left\vert \beta_{\eps} + z \right\vert^2}  \hskip.1cm.$$
We denote by 
$$ e^{\hat{\omega}_{\eps}(x)}  = \eps e^{\tilde{\omega}_{\eps}(\eps x)} $$
the rescaled potential close to the point $(1,0)$ at the scale $\eps$. We see that the metric $ e^{2\hat{\omega}_{\eps}} \xi$ converges to $e^{2u}\xi$ in any compact set of $\mathbb{H}_+$. By symmetry, the rescaled potential close to the point $(-1,0)$ at the scale $\eps$ also converges to $e^{2u}\xi$ on any compact set of $\mathbb{H}_+$. The surface $(\mathbb{D},g_{\eps})$ then represents two flat disks attached by a thin strip.

Now, except for the potentials $e^{\tilde{\omega}_{\eps}}$, and $e^{\hat{\omega}_{\eps}(x)}$ just defined before, we denote for any function $\varphi : \mathbb{D}\to \mathbb{R}$:
$$ \tilde{\varphi}(x) = \varphi \circ f_+^{-1}(x)  \text{ and } \hat{\varphi}(x) = \varphi\circ f_+^{-1}(\eps x) = \tilde{\varphi}(\eps x) $$
Notice that all the functions we consider in the following satisfy that $\varphi \circ f_+^{-1}(x) = \pm \varphi \circ f_-^{-1}(x)$ for any $x$ so that the analysis occuring for eigenfunctions at the neighbourhood of $(1,0)$ is the same as the one occuring at the neighbourhood of $(-1,0)$.

We notice that $L_{g_{\eps}}(\partial\mathbb{D}) = 4\pi$ as a sum of potentials isometric to radius one flat disks. We set $\sigma_{\eps,1}:= \sigma_{1}(\mathbb{D},g_{\eps})$ and $\sigma_{\eps,2}:= \sigma_{2}(\mathbb{D},g_{\eps})$ the first and second eigenvalues associated to $(\mathbb{D},g_{\eps})$. We aim at computing the asymptotic expansion of these quantities as $\eps\to 0$.

\subsection{Upper bounds for first and second eigenvalue} In this part, we aim at proving that
\begin{equation} \label{eqrayleighf1}
\sigma_{\eps,1} \leq \frac{1}{2\ln\frac{1}{\eps} } + O\left( \frac{1}{\left(\ln\frac{1}{\eps}\right)^2}\right)
\end{equation}
\begin{equation} \label{eqrayleighf2}
\sigma_{\eps,2}  \leq 1 +O(\eps)
\end{equation}
as $\eps\to 0$. We set $f_1:\mathbb{D}\to \mathbb{R}$ and $f_2:\mathbb{D}\to \mathbb{R}$ such that
\begin{equation} f_1\circ f_+^{-1}(x) = 
\begin{cases} 1 & \hbox{ on } \mathbb{D}_{\eps}^+ \\
\frac{\ln\left(\left\vert x \right\vert\right)}{\ln \eps} & \hbox{ on } \mathbb{D}_{\frac{1}{\eps}}^+\setminus\mathbb{D}_{\eps}^+ \\
-1 & \hbox{ on } \mathbb{R}^2\setminus \mathbb{D}_{\frac{1}{\eps}}^+ \hskip.1cm.
 \end{cases}  \end{equation}
and
\begin{equation} f_2(z) = \sqrt{\beta_{\eps}^2-1} \frac{z_2}{\left\vert\beta_{\eps}-z\right\vert^2} \hskip.1cm, \end{equation}
defined in order to have that
\begin{equation} f_2\circ f_+^{-1}(\eps x) = \frac{2x_1}{ (1+x_2)^2+x_1^2} \hskip.1cm, \end{equation}
represents the first Steklov eigenfunction $z_2$ of the flat disk. By symmetry, we have that
$$ \int_{\mathbb{S}^1} f_1 d\theta = \int_{\mathbb{S}^1} f_2 d\theta = \int_{\mathbb{S}^1} f_1 f_2 d\theta = 0 \hskip.1cm.$$
By conformal invariance of the Dirichlet energy, we have that
$$ \int_{\mathbb{D}} \left\vert \nabla f_1\right\vert_{g_{\eps}}^2dA_{g_{\eps}}  = \int_{\mathbb{D}}\left\vert \nabla f_1\right\vert^2 = \int_{\mathbb{R}^2_+} \left\vert \nabla f_1\circ f_+^{-1}\right\vert^2 = \pi \int_{\eps}^{\frac{1}{\eps}}\frac{dr}{r\left(\ln\eps\right)^2} = \frac{2\pi}{\ln \frac{1}{\eps}} $$
and we compute the $L^2$ norm on the boundary as
$$ \int_{\mathbb{S}^1}\left(f_1\right)^2dL_{g_{\eps}} = \int_{\mathbb{R}\times\{0\}}\left( \tilde{f}_1\right)^2e^{\tilde{\omega}_{\eps}} = 2 \int_{[-1,1]\times\{0\}}\left( \tilde{f}_1\right)^2e^{\tilde{\omega}_{\eps}} = 2 \int_{-\frac{1}{\eps}}^{\frac{1}{\eps}} \left( \hat{f}_1\right)^2e^{\hat{\omega}_{\eps}}$$
by symmetry and we get
$$ \int_{\mathbb{S}^1}\left(f_1\right)^2dL_{g_{\eps}} = 2\int_{-\frac{1}{\eps}}^{\frac{1}{\eps}} \left( 1-\frac{\ln \left\vert x_1\right\vert}{\ln\eps}\right)^2 \frac{2}{1+\left(x_1\right)^2}dx_1 + O(\eps) = 4\pi + O\left(\frac{1}{\ln\frac{1}{\eps}}\right).$$
Therefore the Rayleigh quotient gives \eqref{eqrayleighf1}.

Again, the conformal invariance of the Dirichlet energy gives that
$$ \int_{\mathbb{D}} \left\vert \nabla f_2\right\vert_{g_{\eps}}^2dA_{g_{\eps}}  = \int_{\mathbb{D}}\left\vert \nabla x_2 \right\vert^2 = \pi$$
and we have by symmetry and neglecting small terms in $e^{\hat{\omega_{\eps}}}$:
$$ \int_{\mathbb{S}^1}\left(f_2\right)^2dL_{g_{\eps}} = 2 \int_{-\frac{1}{\eps}}^{\frac{1}{\eps}} \left( \hat{f}_2\right)^2e^{\hat{\omega}_{\eps}} = 2 \int_{-\frac{1}{\eps}}^{\frac{1}{\eps}} \left(\frac{2x_1}{1+\left(x_1\right)^2}\right)^2 \frac{2}{1+\left(x_1\right)^2}dx_1 + O\left(\eps\right)$$
as $\eps\to 0$. Easy computations by integration by parts finally give
$$ \int_{\mathbb{S}^1}\left(f_2\right)^2dL_{g_{\eps}} = \pi + O\left(\eps\right) $$
as $\eps\to 0$, so that taking the Rayleigh quotient, we get \eqref{eqrayleighf2}.

\subsection{Convergence of the first eigenfunction and the first eigenvalue}

In this subsection, we aim at proving that
\begin{equation}\label{eqasympsigma1} \sigma_{\eps,1} = \frac{1}{2\ln\frac{1}{\eps}} + O\left(\frac{1}{\left(\ln\eps\right)^2}\right)
\end{equation}
as $\eps\to 0$. Let $\varphi_{\eps,1}$ be a first eigenfunction associated to $\sigma_{1,\eps}$. We have the following equation
\begin{equation}
\label{eqsteklovtonvarpuieps}
\begin{cases}
\Delta \varphi_{\eps,1} = 0 & \text{ in } \mathbb{D} \\
\partial_{r} \varphi_{\eps,1} = \sigma_{\eps,1}e^{\omega_{\eps}} \varphi_{\eps,1} & \text{ in } \mathbb{S}^1
\end{cases}
\text{ and }
\int_{\mathbb{S}^1}\left(\varphi_{\eps,1}\right)^2dL_{g_{\eps}} = 4\pi \hskip.1cm.
\end{equation}
Up to symmetrize or antisymmetrize in case of multiplicity (that a posteriori do not happen), we can assume that
\begin{equation}\label{eqsymassumpsteklov1} \varphi_{\eps,1}(z_1,z_2)^2 = \varphi_{\eps,1}(-z_1,z_2)^2 = \varphi_{\eps,1}(z_1,-z_2)^2 \end{equation}
for any $z=(z_1,z_2)\in \mathbb{D}$. Then $\hat{\varphi}_{\eps,1}$ satisfies the following equation in $\mathbb{R}^2_+$:
\begin{equation}
\label{eqsteklovtonvarpuieps+}
\begin{cases}
\Delta \hat{\varphi}_{\eps,1} = 0 & \text{ in } \mathbb{D} \\
-\partial_{x_2} \hat{\varphi}_{\eps,1} = \sigma_{\eps,1}\frac{2}{1+(x_1)^2} \hat{\varphi}_{\eps,1}+ \sigma_{\eps,1}\frac{2\eps^2}{1+(x_1)^2\eps^4} \hat{\varphi}_{\eps,1} & \text{ on } \mathbb{R}\times\{0\}
\end{cases}
\end{equation}
and
$$
\int_{\mathbb{R}\times\{0\}}\left(\hat{\varphi}_{\eps,1}\right)^2 \left(\frac{2}{1+(x_1)^2}  +\frac{2\eps^2}{1+(x_1)^2\eps^4} \right)dx_1 = 4\pi \hskip.1cm.
$$
In particular, $
\int_{\mathbb{R}\times\{0\}}\left(\hat{\varphi}_{\eps,1}\right)^2 \frac{2}{1+(x_1)^2} dx_1 \leq 4\pi $ and by standard Elliptic theory, since $\sigma_{\eps,1}\to 0$ as $\eps\to 0$, we obtain that
$$\hat{\varphi}_{\eps,1}\to \varphi_{\star,1}^+ \text{ in } \mathcal{C}^2\left(\mathbb{D}_{\rho}\right) $$
as $\eps\to 0$, for any $\rho>0$. Letting \eqref{eqsteklovtonvarpuieps+} pass to the limit as $\eps\to 0$, we obtain that 
\begin{equation}
\label{eqsteklovtonvarpuistar+}
\begin{cases}
\Delta \hat{\varphi}_{\star,1}^+ = 0 & \text{ in } \mathbb{R}^2_+ \\
-\partial_{x_2} \hat{\varphi}_{\star,1}^+ = 0 & \text{ on } \mathbb{R}\times\{0\}
\end{cases}
\end{equation}
Since $\int_{\mathbb{R}\times\{0\}}\left(\hat{\varphi}_{\eps,1}\right)^2 \frac{2}{1+(x_1)^2} dx_1 \leq 4\pi$, we get that $\varphi_{\star}^+$ is a constant function. Up to take $-\varphi_{\eps,1}$ instead of $\varphi_{\eps,1}$, we can assume that $\varphi_{\star,1}^+\geq 0$. Looking at $\varphi_{\eps,1}\circ \left(f_{-}\right)^{-1}(\eps x)$ instead of $\hat{\varphi}_{\eps,1}(x):=\varphi_{\eps,1}\circ \left(f_{+}\right)^{-1}(\eps x)$, we also get a constant function $\varphi_{\star,1}^-$ at the limit and by symmetry, $\left(\varphi_{\star,1}^+\right)^2 = \left(\varphi_{\star,1}^-\right)^2$.

\medskip

We denote by
$$ m_{\eps,1}(r) = \frac{1}{\pi}\int_{0}^\pi \tilde{\varphi}_{\eps,1}(r\cos \theta,r\sin\theta)d\theta $$
the mean value of $\tilde{\varphi}_{\eps,1}$ on half circle (corresponding to the mean value of $\varphi_{\eps,1}$ on the lines $\left\vert \frac{1-z}{1+z}\right\vert = constant$). It is a radial function. We also set 
$$ \psi_{\eps,1}(x) = \tilde{\varphi}_{\eps,1}(x) - m_{\eps,1}(\left\vert x \right\vert) \hskip.1cm.$$
We have
\begin{equation}
\label{eqsteklovtonpsieps}
\begin{cases}
\Delta \psi_{\eps,1} = - \Delta m_{\eps,1} = -\frac{\sigma_{\eps,1}}{\pi r}e^{\tilde{\omega}_{\eps}}\left(\tilde{\varphi}_{\eps,1}(r,0)+\tilde{\varphi}_{\eps,1}(-r,0)\right) & \text{ in } \mathbb{R}^2_+ \\
-\partial_{x_2} \psi_{\eps,1} = \sigma_{\eps,1}e^{\tilde{\omega}_{\eps}}\tilde{\varphi}_{\eps,1} & \text{ on } \mathbb{R}\times\{0\} \hskip.1cm.
\end{cases}
\end{equation}
Let $x_{\eps}\in \mathbb{D}_{+}$ be such that $\psi_{\eps}(x_{\eps}) = \left\|\psi_{\eps,1} \right\|_{\infty}$. We aim at proving that
\begin{equation} \label{ineqpsiepsunifboundedsteklov} \left\|\psi_{\eps,1} \right\|_{\infty}\leq C\sigma_{\eps,1} \end{equation}
for some positive constant $C$. We have two cases:

If $\left\vert x_{\eps} \right\vert = O\left(\eps\right)$, then $\hat{\psi}_{\eps,1}$ satisfies 
\begin{equation}
\label{eqsteklovtonpsieps+}
\begin{cases}
\Delta \hat{\psi}_{\eps,1} =  -\frac{\sigma_{\eps,1}}{\pi r}\left(\frac{2}{1+r^2}+\frac{2\eps^2}{1+\eps^4r^2}\right)\left(\hat{\varphi}_{\eps,1}(r,0)+\hat{\varphi}_{\eps,1}(-r,0)\right) & \text{ in } \mathbb{R}^2_+ \\
-\partial_{x_2} \hat{\psi}_{\eps,1} = \sigma_{\eps,1}\frac{2}{1+(x_1)^2}\hat{\varphi}_{\eps,1} + \sigma_{\eps,1}\frac{2\eps^2}{1+\eps^4(x_1)^2}\hat{\varphi}_{\eps,1}  & \text{ on } \mathbb{R}\times\{0\} \hskip.1cm.
\end{cases}
\end{equation}
We also know by assumption that for any $\rho>0$ we have $ \int_{\mathbb{D}_\rho^+} \psi_{\eps,1}= 0 $. We also know that $\psi_{\eps,1}$ has a bounded energy. By asumption, the right-hand side term of the second equation is bounded in $L^2\left(\mathbb{R}^2_+\right)$. Moreover, by H\"older inequalities, the right-hand side term in the first equation is bounded in $L^p$ for any $p<2$. By standard elliptic theory \eqref{ineqpsiepsunifboundedsteklov} holds true.

If $\eps = o\left(\left\vert x_{\eps} \right\vert\right) $, then we set $\phi_{\eps}(x)= \psi_{\eps,1}(\left\vert x_{\eps} \right\vert x)$ and we have the equation
\begin{equation}
\label{eqsteklovtonphieps+}
\begin{cases}
\Delta \phi_{\eps} =  -\frac{\sigma_{\eps,1}}{\pi r}\left(\frac{2\frac{\eps}{\left\vert x_{\eps} \right\vert}}{\frac{\eps^2}{\left\vert x_{\eps} \right\vert^2}+r^2}+\frac{2\eps^2\left\vert x_{\eps} \right\vert}{1+\eps^4\left\vert x_{\eps} \right\vert^2 r^2}\right)\left(\tilde{\varphi}_{\eps,1}(r\left\vert x_{\eps} \right\vert,0) + \tilde{\varphi}_{\eps,1}(-r \left\vert x_{\eps} \right\vert,0)\right) \\
-\partial_{x_2} \phi_{\eps} =   \frac{\sigma_{\eps,1}.2\frac{\eps}{\left\vert x_{\eps} \right\vert}}{\frac{\eps^2}{\left\vert x_{\eps} \right\vert^2}+(x_1)^2}\tilde{\varphi}_{\eps,1}(-x_1\left\vert x_{\eps}\right\vert,0) + \frac{\sigma_{\eps,1}.2\eps^2\left\vert x_{\eps} \right\vert}{1+\eps^4\left\vert x_{\eps} \right\vert^2(x_1)^2}\tilde{\varphi}_{\eps,1}(x_1\left\vert x_{\eps}\right\vert,0) 
\end{cases}
\end{equation}
so that by standard elliptic estimates in $A = \mathbb{D}_2^+\setminus \mathbb{D}_{\frac{1}{2}}^+$, knowing that $\int_{A} \phi_{\eps} = 0$, we obtain thanks to $L^p$ boundedness in the left-hand side terms of the equation for $p<2$ that $\left\|\psi_{\eps} \right\|_{\infty} \leq C \sigma_1^{\eps}$ for some positive constant $C$ independent from $\eps$ and \eqref{ineqpsiepsunifboundedsteklov} holds true.

\medskip

We deduce from \eqref{ineqpsiepsunifboundedsteklov}  that
\begin{equation} \label{eqL2normfirsteigenfunction}
4\pi = \int_{\mathbb{S}^1} \left(\varphi_{\eps,1}\right)^2dL_{g_{\eps}} = \int_{\mathbb{S}^1} \left(m_{\eps,1}\right)^2dL_{g_{\eps}} + O\left(\sigma_{\eps,1}\right)
 \end{equation}
as $\eps\to 0$.

By the Courant nodal theorem, and symmetry assumptions \ref{eqsymassumpsteklov1}, the only possibilities for nodal sets are $\{z_1 = 0\}$ or $\{z_2=0\}$. But if $\{z_2  = 0\}$ is a nodal set, we obtain that $\varphi_{\eps,1}(z_1,z_2) = \varphi_{\eps,1}(z_1,-z_2) $ by the Courant nodal theorem and we would get $m_{\eps,1}=0$. However, since $\sigma_{\eps,1}\to 0$ as $\eps\to 0$, \eqref{eqL2normfirsteigenfunction} gives a contradiction. Therefore, $\{z_1 = 0\}$ is the nodal line of $\varphi_{\eps,1}$ and $\{\left\vert x \right\vert = 1\}$ is the zero set for $m_{\eps,1}$. Moreover $m_{\eps,1}$ and $\tilde{\varphi}_{\eps,1}$ are positive in $\mathbb{D}$ and negative in $\mathbb{R}^2\setminus\mathbb{D}$. $m_{\eps,1}$ satisfies:
$$ \Delta m_{\eps,1}:= -\frac{1}{r}\partial_r\left(rm_{\eps,1}'\right) = \frac{\sigma_{\eps,1}}{\pi r} e^{\tilde{\omega}_{\eps}}\left(\tilde{\varphi}_{\eps,1}(r,0)+\tilde{\varphi}_{\eps,1}(-r,0)\right) \hskip.1cm.$$
Integrating between $0$ and $r\leq 1$, we get
\begin{equation}\label{eqintegratesteklov1} m_{\eps,1}(r)-m_{\eps,1}(0) =  -\int_{0}^r \frac{1}{s}\left( \int_{[-s,s]\times\{0\}} \sigma_{\eps,1}e^{\tilde{\omega_{\eps}}}\tilde{\varphi}_{\eps,1} \right) \leq 0 \end{equation}
Then, $m_{\eps,1}$ realizes its maximum at $0$. Then, we have that for any $\rho>0$,
$$ \int_{\rho\eps}^{1} \left(\tilde{\varphi}_{\eps,1}\right)^2 e^{\tilde{\omega}_{\eps}} dx_1 \leq\left(\varphi_{\eps,1}(0)\right)^2\int_{\rho\eps}^{1}  e^{\tilde{\omega}_{\eps}} dx_1 \leq \frac{C}{\rho} $$
for a positive constant $C$ independent from $\rho$ and $\eps$. Letting $\eps\to 0$, and then $\rho \to 0$ in \eqref{eqL2normfirsteigenfunction}, we obtain that $\left(\varphi_{\star,1}^+\right)^2+\left(\varphi_{\star,1}^-\right)^2 = 2$. Then, since we know that $\varphi_{\eps,1}(z_1,z_2) = -\varphi(-z_1,z_2)$, and $\varphi_{\star,1}^+ \geq 0$, we have that $\varphi_{\star,1}^+ = 1$ and $\varphi_{\star,1}^- = -1$.

Taking $r=1$ in \eqref{eqintegratesteklov1}, we obtain
\begin{eqnarray*} m_{\eps,1}(0) &= & \frac{\sigma_{\eps,1}}{\pi} \int_{0}^1\frac{1}{s}\left( \int_{[-s,s]\times\{0\}} e^{\tilde{\omega}_{\eps}}\tilde{\varphi}_{\eps,1} \right)   \\
& \leq & \frac{\sigma_{\eps,1}}{\pi} \tilde{\varphi}_{\eps,1}(0) \int_{0}^{1}\frac{1}{s} \int_{-s}^{s}  \frac{2\eps}{\eps^2+u^2}   du + O\left(\eps \right) \\
& \leq & \frac{\sigma_{\eps,1}}{\pi} \tilde{\varphi}_{\eps,1}(0) \int_{0}^{\frac{1}{\eps}} \frac{2 }{1 +r^2 }  \left(\ln \frac{1}{ r\eps} \right)dr + O\left(\eps \right) \\
& \leq & \frac{\sigma_{\eps,1}}{\pi} \tilde{\varphi}_{\eps,1}(0)\left(\ln\frac{1}{\eps}\right) \int_{0}^{+\infty} \frac{2}{1+r^2} dr + O\left( \sigma_{\eps,1}\right) + O\left(\eps \right) \\ \end{eqnarray*}
and knowing that $\tilde{\varphi}_{\eps,1}(0)\to 1$ as $\eps\to 0$, we get 
$$ \sigma_{\eps,1} \geq \frac{1}{2\ln\frac{1}{\eps}} + O\left(\frac{\sigma_{\eps,1}}{\ln\frac{1}{\eps}}\right) $$
so that with \eqref{eqrayleighf1}, we obtain \eqref{eqasympsigma1}.

\subsection{The second eigenfunction and second eigenvalue}
In this section, we aim at proving that 
\begin{equation}\label{eqprovesigma2eps} \sigma_{\eps,2} = 1 - 4\eps + o(\eps) \end{equation}
as $\eps\to 0$.

We focus on the equation satisfied by $\varphi_{\eps,2}$, an eigenfunction associated to the second non-zero eigenvalue $\sigma_{\eps,2}$, satisfying 
\begin{equation}\label{eqL2valuessecondsteklov} \int_{\mathbb{S}^1} \left(\varphi_{\eps,2}\right)^2 e^{\omega_{\eps}}dz = \pi \hbox{ and } \int_{\mathbb{S}^1} \varphi_{\eps,2} e^{\omega_{\eps}}dz = \int_{\mathbb{S}^1} \varphi_{\eps,1}\varphi_{\eps,2} e^{\omega_{\eps}}dz = 0 \hskip.1cm,\end{equation}
where $\varphi_{\eps,1}$ is the radial first eigenfunction of the previous section. Up to symmetrization again, in case of multiplicity, we assume in addition that 
\begin{equation} \label{symmetriessecondeigenfunctionsteklov} \varphi_{\eps,2}(z_1,z_2)^2 = \varphi_{\eps,2}(-z_1,z_2)^2 = \varphi_{\eps,2}(z_1,-z_2)^2 \end{equation}
In fact, we must have 
\begin{equation} \label{symmetrysecondeigenfunctionsteklov} \varphi_{\eps,2}(z_1,z_2) = \varphi_{\eps,2}(-z_1,z_2) \hskip.1cm. \end{equation} 
Indeed, if not, we would have antisymmetry so that $\varphi_{\eps,2}(0,z_2) = 0$. By orthogonality with $\varphi_{\eps,1}$ which is anti-symmetric, $\varphi_{\eps,2}$ must vanish elsewhere: $\varphi_{\eps,2}(z_1,z_2) = 0$ for $z_1> 0$. The nodal line containing $z$ is a line connecting two points of $\left(\mathbb{S}^1\cap\{z_1>0\}\right) \cup \{z_1=0\}$.  Then $\varphi_{\eps,2}$ has at least two nodal domains in $\{z_1>0\}$. By symmetry, $\varphi_{\eps,2}$ has at least four nodal domains, contradicting the Courant nodal theorem. We also deduce from the orthogonality with the constants that
\begin{equation} \label{eqorthogonalitysteklovwithsymmetry2} \int_{\mathbb{S}^1\cap\{z_1>0\}} \varphi_{\eps,2} e^{\omega_{\eps}}dz = \int_{\mathbb{S}^1\cap \{z_1<0\}} \varphi_{\eps,2} e^{\omega_{\eps}}dz = 0 \hskip.1cm. \end{equation}

\medskip 

Rescaling at the neighbourhood of $(1,0)$, $\hat{\varphi}_{\eps,2}$ satisfies the equation
\begin{equation}
\label{eqsteklovtonvarpuieps+2}
\begin{cases}
\Delta \hat{\varphi}_{\eps,2} = 0 & \text{ in } \mathbb{D} \\
-\partial_{x_2} \hat{\varphi}_{\eps,2} = \sigma_{\eps,2}\frac{2}{1+(x_1)^2} \hat{\varphi}_{\eps,2}+ \sigma_{\eps,2}\frac{2\eps^2}{1+(x_1)^2\eps^4} \hat{\varphi}_{\eps,2} & \text{ on } \mathbb{R}\times\{0\}
\end{cases}
\end{equation}
with 
$$ \int_{\mathbb{R}\times\{0\}} \left( \frac{2}{\left(1+( x_1)^2\right)^2}  + \frac{2\eps^2}{1+\eps^4( x_1)^2}\right) \left(\hat{\varphi}_{\eps,2}\right)^2 dx = \pi \hskip.1cm.$$
In particular, $\int_{\mathbb{R}\times\{0\}}  \frac{2}{\left(1+( x_1)^2\right)^2} \left(\hat{\varphi}_{\eps,2}\right)^2 dx \leq \pi $. $\sigma_{\eps,2}$ is bounded by \eqref{eqrayleighf2} and converges up to the extraction of a subsequence to $\sigma_{\star,2}$. By standard elliptic estimates, up to the extraction of a subsequence,
\begin{equation} \label{equnifconvpgieps2steklov} \hat{\varphi}_{\eps,2} \to \varphi_{\star,2}^+ \hbox{ in } \mathcal{C}^{2}\left(\mathbb{D}_\rho^+ \right) \end{equation}
for any $\rho>0$ and
\begin{equation}
\label{eqsteklovtonvarpuistar+2}
\begin{cases}
\Delta \varphi_{\star,2}^+  = 0 & \text{ in } \mathbb{R}^2_+ \\
-\partial_{x_2} \varphi_{\star,2}^+ = \sigma_{\star,2}e^u \varphi_{\star,2}^+  & \text{ on } \mathbb{R}\times\{0\}
\end{cases}
\end{equation}
By boundedness of the energy, $\varphi_{\star,2}^+ \circ f_+$ has to be a Steklov eigenfunction associated to $\sigma_{\star,2}$ on the disk. Since $\sigma_{\star,2}\leq 1$ by \eqref{eqrayleighf2}, we have either $\sigma_{\star,2} = 0$ or $\sigma_{\star,2} =1$. We aim at proving that $\sigma_{\star,2} = 1$ and at getting an estimate on $\delta_{\eps} := \sigma_{\eps,2}-1$ as $\eps\to 0$.

\medskip

Let's deal with the second symmetry in \eqref{symmetriessecondeigenfunctionsteklov}. In the following subsections we separate two cases:
\begin{equation}\label{eqantisym} \varphi_{\eps,2}(z_1,z_2) = -\varphi_{\eps,2}(z_1,-z_2) \end{equation}
is called the antisymmetric case and is handled in subsection \ref{subsectionantisymmetriccase} and in this case, we assume in addition that $\varphi_{\eps,2}(0,1)\geq 0$ up to take $-\varphi_{\eps,2}$.
\begin{equation}\label{eqsym} \varphi_{\eps,2}(z_1,z_2) = \varphi_{\eps,2}(z_1,-z_2) \end{equation}
is called the symmetric case and is handled in subsection \ref{subsectionsymmetriccase}, and in this case, we assume in addition that $\varphi_{\eps,2}(1,0)\geq 0$ up to take $-\varphi_{\eps,2}$. We a posteriori prove that this case cannot occur for the second eigenfunction.

\subsubsection{The antisymmetric case} \label{subsectionantisymmetriccase}
By antisymmetry \eqref{eqantisym}, we must have that for any $r>0$,
$$ \int_0^{\pi} \tilde{\varphi}_{\eps,2}(r\cos\theta,r\sin\theta)d\theta = 0$$
so that we can handle elliptic estimates in $\tilde{\varphi}_{\eps,2}$ at any scales. In particular we have that $\tilde{\varphi}_{\eps,2}$ is uniformly bounded. We then obtain that
\begin{equation}\label{eqconcentrationenergysteklov} \int_{\rho\eps}^1 \tilde{\varphi}_{\eps,2} e^{\tilde{\omega}_{\eps}}dx_1 \leq \frac{C}{\rho} \end{equation}
for a constant $C>0$ independant from $\rho>0$ and $\eps>0$. By \eqref{eqL2valuessecondsteklov}, \eqref{eqorthogonalitysteklovwithsymmetry2} and \eqref{equnifconvpgieps2steklov} we obtain that
$$ \int_{\mathbb{R}\times\{0\}} \left(\varphi_{\star,2}^+\right)^2e^u dx_1 = \frac{\pi}{2} \text{ and } \int_{\mathbb{R}\times\{0\}} \varphi_{\star,2}^+e^u dx_1 = 0$$ 
By \eqref{eqsteklovtonvarpuistar+2}, and remarks below, we must have $\sigma_{\star,2}=1$ and $\varphi_{\star,2}^+ \circ f_+$ has to be a first Steklov eigenfunction on the disk, that is
$$ \varphi_{\star,2}^+ = a_{\star,1} Z_1 +a_{\star,2} Z_2 \hskip.1cm,$$
where $Z_1 = Re(f_+^{-1})$ and $Z_2 = Im(f_+^{-1})$ and $\left(a_{\star,1} \right)^2 +\left(a_{\star,2}\right)^2=1$. By antisymmetry \eqref{eqantisym}, $a_{\star,1}=0$ and $a_{\star,2}=1$.

\medskip

Now, let's work on $\delta_{\eps} = \sigma_{\eps,2}-1$.

\medskip

We define $R_{\eps}:\mathbb{D}\to \mathbb{R}$ such that the function $\tilde{R}_{\eps} = R_{\eps}\circ f_{+}^{-1}$ satisfies
\begin{equation} \label{definitionrepstildesteklov} \tilde{R}_{\eps}(x)= \varphi_{\eps,2}(x) - \eta_{\eps}\left(Z_2\left(\frac{x}{\eps}\right)+Z_2\left(x\eps\right) \right)
\end{equation}
and such that $R_{\eps}\circ f_{-}^{-1}$ satisfies the same equation, where $\eta_{\eps}$ is defined in order to have $\nabla R_{\eps}(-1,0) = \nabla R_{\eps}(1,0) = 0$. Such a $\eta_{\eps}$ exists because of anti-symmetry of the involved functions. Notice that $\eta_{\eps} = 1+o(1)$ as $\eps\to 0$. We obtain the following equation on $R_{\eps}$
\begin{equation}\label{eqRepssteklov}
\begin{cases}
\Delta R_{\eps}=0 & \text{ in } \mathbb{D} \\
\partial_r R_{\eps} - e^{\omega_{\eps}}R_{\eps} = (\sigma_{\eps,2}-1)e^{\omega_{\eps}}\varphi_{\eps,2} + \eta_{\eps}\left(e^{u_{\eps}}Z_{2,\eps}^+ + e^{v_{\eps}}Z_{2,\eps}^- \right) & \text{ on } \mathbb{S}^1
\end{cases}
\end{equation}
where $Z_{2,\eps}^{\pm} = Z_2\left(\frac{f_{\pm}^{-1}(x)}{\eps}\right)$ and then integrating against $R_{\eps}$ in $\mathbb{D}$,
\begin{equation}\label{eqW12estonRepsantisym}
\begin{split}
\int_{\mathbb{D}} \left\vert \nabla R_{\eps}\right\vert^2 =&  2 \int_{\mathbb{S}^1} \left(R_{\eps}\right)^2 e^{u_{\eps}} + \left(\sigma_{\eps,2}-1\right) \int_{\mathbb{S}^1}e^{\omega_{\eps}} \varphi_{\eps,2}R_{\eps} \\
&+ \eta_{\eps} \int_{\mathbb{S}^1} \left(e^{u_{\eps}} Z_{2,\eps}^+  + e^{v_{\eps}} Z_{2,\eps}^- \right) R_{\eps} \\
\leq & \left\| R_{\eps} \right\|_{\infty} \left( \left\| R_{\eps} \right\|_{\infty} +2\pi \left\| \varphi_{\eps,2} \right\|_{\infty} \left\vert \sigma_{\eps,2} -1 \right\vert + O\left(\eps\right) \right) \\
\leq & C \left\| R_{\eps} \right\|_{\infty} \left(\left\| R_{\eps} \right\|_{\infty} + \delta_{\eps} + \eps  \right) \hskip.1cm,
\end{split}
\end{equation}
where we computed $I:= \int_{\mathbb{S}^1} \left(e^{u_{\eps}} Z_{2,\eps}^+  + e^{v_{\eps}} Z_{2,\eps}^- \right) R_{\eps}$ as
\begin{equation*} 
\begin{split}
I = & 2 \int_{\mathbb{S}^1_-} \left(e^{u_{\eps}} Z_{2,\eps}^+  + e^{v_{\eps}} Z_{2,\eps}^- \right) R_{\eps} \\
= & 2 \int_{-1}^{1} \left(\frac{2\eps}{1+ \eps^2 ( x_1 )^2  } \frac{2x_1 \eps}{\eps^2 +  (x_1)^2 } + \frac{2\eps}{\eps^2+ (x_1 )^2  } \frac{2x_1 \eps}{1 + \eps^2 ( x_1 )^2 } \right) \tilde{R}_{\eps} dx_1\\
\leq & 4 \left\| R_{\eps} \right\|_{\infty} \eps \int_0^{\frac{1}{\eps}} \left( \frac{2\eps}{1+ \eps^4 u^2 } \frac{2 u }{1 + u^2 } + \frac{2}{1+ u^2 } \frac{2u}{1 + \eps^4 u^2 }  \right) du \\
\leq & O\left( \left\| R_{\eps} \right\|_{\infty} \eps \right)
\end{split}
\end{equation*}
as $\eps\to 0$. We set
\begin{equation}\label{defalphaepsantisym} \alpha_{\eps} = \left\| R_{\eps} \right\|_{L^\infty\left(\mathbb{S}^1\right)} + \left\vert\delta_{\eps}\right\vert + \eps \hskip.1cm,\end{equation}
and we get that 
\begin{equation} \label{eqW12estonRepsantisym2}
\int_{\mathbb{D}} \left\vert \nabla R_{\eps}\right\vert^2 \leq O\left( \alpha_{\eps}^2 \right)
\end{equation}
as $\eps\to 0$. Letting $\hat{R}_{\eps}(x) = \tilde{R}_{\eps}(\eps x)$, we get from \eqref{eqRepssteklov} that
\begin{equation}\label{eqRepssteklov+}
\begin{cases}
\Delta \hat{R}_{\eps}= 0 &  \text{ in }  \mathbb{R}^2_+ \\
\begin{split} -\partial_{x_2} \hat{R}_{\eps} - e^{u}\hat{R}_{\eps} =& (\sigma_{\eps,2}-1)e^{u}
\hat{\varphi}_{\eps,2} + \eta_{\eps} e^u Z_2( \eps^2 x)  \\
&+ \frac{2\eps^2}{1+\eps^4 (x_1)^2}\left(\sigma_{\eps,2}\hat{\varphi}_{\eps,2} - \eta_{\eps}Z_2(\eps^2 x)   \right) \end{split} & \text{ on } \mathbb{R}\times\{0\} \hskip.1cm.
\end{cases}
\end{equation}
Dividing this equation by $\alpha_{\eps}$, by standard elliptic theory, we obtain up to the extraction of a subsequence that
$$ \frac{R_{\eps}^+}{\alpha_{\eps}} \to R_{\star}^+ \in \mathcal{C}^2\left(\mathbb{D}_{\rho}\right) $$
as $\eps\to 0$ for any $\rho>0$. We also have that up to the extraction of a subsequence, 
$$ \frac{\delta_{\eps}}{\alpha_{\eps}}\to \delta_{\star} \text{ and } \frac{\eps}{\alpha_{\eps}}\to e_{\star} \text{ as } \eps\to 0 \hskip.1cm. $$
Passing to the limit in \eqref{eqRepssteklov+} divided by $\alpha_{\eps}$, $R_{\star}^+$ satisfies the equation
\begin{equation}\label{eqRstarsteklov+}
\begin{cases}
\Delta R_{\star}^+= 0 & \text{ in } \mathbb{R}^2_+ \\
 -\partial_{x_2} R_{\star}^+ - e^{u}R_{\star}^+ = \delta_{\star} e^{u}
Z_2 + 2 e^u x_1 e_{\star} & \text{ on } \mathbb{R}\times\{0\} \hskip.1cm.
\end{cases}
\end{equation}
We set $R_{\star}:= R_{\star}^+\circ f_+$ in $\mathbb{D}\setminus\{(-1,0)\}$ and we obtain
\begin{equation}\label{eqRstarsteklov}
\begin{cases}
\Delta R_{\star} = 0 & \text{ in } \mathbb{D} \\
 -\partial_{x_2} R_{\star} - e^{u}R_{\star} = \left(\delta_{\star}
+ e_{\star}\frac{2}{\left\vert 1+ z \right\vert^2} \right) z_2 & \text{ on } \mathbb{S}^1\setminus\{(-1,0)\} \hskip.1cm.
\end{cases}
\end{equation}
Notice that the left-hand side in the secons equation is uniformly bounded and that by \eqref{eqW12estonRepsantisym2}, $R_{\star}\in W^{1,2}\left(\mathbb{D}\right)$ so that $R_{\star}$ can be extended in $\mathbb{D}$ such that the equation holds in $\mathbb{D}$. Integrating this equation against $z_2$, we obtain that 
$$ \delta_{\star}  = -\frac{2\int_{\mathbb{S}^1}\frac{(z_2)^2}{\left\vert 1+ z \right\vert^2}d\theta}{\int_{\mathbb{S}^1} (z_2)^2 d\theta} e_{\star} = -\frac{2\pi}{\frac{\pi}{2}} e_{\star} \hskip.1cm. $$
Notice that if $e_{\star} \neq 0$, then
\begin{equation}
\label{eqlimdeltaepssteklov}
\frac{\delta_{\eps}}{\eps} = \frac{\frac{\delta_{\eps}}{\alpha_{\eps}}}{\frac{\eps}{\alpha_{\eps}}} \to \frac{\delta_{\star}}{e_{\star}} = - 4 \hskip.1cm,
\end{equation}
and \eqref{eqprovesigma2eps} would be proved in the antisymmetric case.

\medskip

From now to the end of the subsection \ref{subsectionantisymmetriccase}, we assume by contradiction that $e_{\star}=0$. This implies that $\delta_{\star}=0$. By \eqref{defalphaepsantisym}, and definition of $\delta_{\star}$ and $e_{\star}$, we get
\begin{equation} \label{eqestimateondeltaepsantisym} \delta_{\eps} = o\left( \left\| R_{\eps} \right\|_{\infty}  \right) \text{ and } \eps = o\left( \left\| R_{\eps} \right\|_{\infty}  \right)
\end{equation}
as $\eps\to 0$. Moreover, \eqref{eqRstarsteklov+} becomes $\Delta R_{\star} = 0$ and $-\partial_{x_2}R_{\star}-R_{\star} = 0$. Then $R_{\star}$ is a first Steklov eigenfunction in $\mathbb{D}$ satisfying $R_{\star}(-1,0) = 0$ and $\nabla R_{\star}(-1,0) = 0$ since $\tilde{R}_{\eps,2}(0) = 0$ and $\nabla \tilde{R}_{\eps,2}(0) = 0$ held for any $\eps>0$. Therefore, $R_{\star} = 0$ and we obtain
\begin{equation} \label{eqestimateonRepsantisym} R_{\eps}(\rho \eps) = o\left( \left\| R_{\eps} \right\|_{\infty}  \right)
\end{equation}
as $\eps\to 0$, for any $\rho>0$. 

Let $x_{\eps}\in \mathbb{D} \cap \left(\mathbb{R}\times\{0\}\right)$ be such that $R_{\eps}(x_\eps) =\left\| R_{\eps} \right\|_{\infty}$. We obtain from \eqref{eqestimateonRepsantisym} that $\eps = o\left(r_{\eps}\right)$, letting $r_{\eps} = \left\vert x_\eps \right\vert$. We set $\psi_{\eps}(x) = R_{\eps}(r_\eps x)$ and $\phi_{\eps}(x) = \tilde{\varphi}_{\eps,2}(r_{\eps}x)$ and we obtain from \eqref{eqRepssteklov+}
\begin{equation}\label{eqpsiepssteklov+}
\begin{cases}
\Delta \psi_{\eps}= 0 &  \text{ in }  \mathbb{R}^2_+ \\
\begin{split} -\partial_{x_2} \psi_{\eps} = & \left(  \psi_{\eps} + \left(\sigma_{\eps,2}-1\right) \phi_{\eps} + \eta_{\eps}Z_2(r_{\eps}\eps x) \right)  \frac{2\frac{\eps}{r_{\eps}}}{\frac{\eps^2}{r_{\eps}^2} +( x_1 )^2} \\
 &+ \left(\sigma_{\eps,2}\phi_{\eps}-\eta_{\eps}Z_2(r_{\eps}\eps x)\right) \frac{2\eps r_{\eps}}{1+\eps^2r_{\eps}^2( x_1)^2}  \end{split} & \text{ on } \mathbb{R}\times\{0\} \hskip.1cm.
\end{cases}
\end{equation}
so that dividing by $\alpha_{\eps}$, we get that for any $\rho>1$, 
\begin{equation*}
\begin{split} \left\|  \partial_{x_2} \psi_{\eps}\right\|_{L^\infty\left(\left([-\rho,\rho]\setminus[-\frac{1}{\rho}, \frac{1}{\rho}]\right)\times \{0\}\right)} 
= O\left( \left(\left\| \psi_{\eps}\right\|_{\infty}  + \left\vert \delta_{\eps} \right\vert\right) \frac{\eps}{r_{\eps}}\right) + o(\eps)
\\\ = O\left( \left\| \psi_{\eps}\right\|_{\infty}\frac{\eps}{r_{\eps}}\right) +o(\eps) \hskip.1cm.
 \end{split} \end{equation*}
By standard elliptic theory, we obtain up to the extraction of a subsequence that 
$$\frac{\psi_{\eps}}{\alpha_{\eps}} \to \psi_{\star} \text{ in } \mathcal{C}^2\left(\mathbb{D}_\rho^+\setminus \mathbb{D}_{\frac{1}{\rho}}^+\right)$$
as $\eps\to 0$ for any $\rho>0$.  We also define up to the extraction of a subsequence
$$ \frac{x_{\eps}}{r_{\eps}} \to x_{\star}\in \{(\pm 1,0)\} \hbox{ as } \eps\to 0 \hskip.1cm. $$
Passing to the limit in \eqref{eqpsiepssteklov+} divided by $\alpha_{\eps}$, $\psi_{\star}$ is a harmonic function on $\mathbb{R}^2_+ \setminus\{0\}$ satisfying $\partial_{x_2}\psi_{\star}=0$. Since $\psi_{\star}$ is bounded by $1$, $\psi_{\star}$ is a constant function. But since the mean value of $\psi_{\star}$ is equal to $0$ on any circle centered at $0$, we obtain that $\psi_{\star}(x_{\star}) = 0$. Therefore $\left\| R_{\eps} \right\|_{L^\infty\left(\mathbb{S}_1\right)} = o\left(\alpha_{\eps}\right)$ and the definition of $\alpha_{\eps}$ \eqref{defalphaepsantisym} and \eqref{eqestimateondeltaepsantisym} give a contradiction.

\medskip

Therefore, $e_{\star}\neq 0$ and we proved \eqref{eqprovesigma2eps} thanks to \eqref{eqlimdeltaepssteklov} in the anti-symmetric case.

\subsubsection{The symmetric case} \label{subsectionsymmetriccase}
We assume \eqref{eqsym}. We aim at proving that $\delta_{\eps,2}:= \sigma_{\eps,2}-1 = o(\eps)$ as $\eps\to 0$, so that $\varphi_{\eps,2}$ cannot be a second eigenfunction by the previous case and it is a contradiction, and then, only \eqref{eqlimdeltaepssteklov} occurs.

Since $\varphi_{\eps,2}$ has to vanish somewhere but has at most three nodal domains, the symmetries \eqref{symmetrysecondeigenfunctionsteklov} and \eqref{eqsym} and the orthogonality with the first eigenfunction imply that there is a unique value $r_{\eps}\in (0,1)$ such that $ \tilde{\varphi}_{\eps,2}(r_{\eps},0) =  \tilde{\varphi}_{\eps,2}(-r_{\eps},0) = 0$. Morover $\tilde{\varphi}_{\eps,2}$ is positive on $(-r_{\eps},r_{\eps})\times\{0\}$ and negative on $\left([-1,1]\setminus[-r_{\eps},r_{\eps}]\right)\times\{0\}$.

\medskip

As for the first eigenfunction, we set
$$ m_{\eps,2}(r) = \frac{1}{\pi}\int_{0}^\pi \tilde{\varphi}_{\eps,2}(r\cos \theta,r\sin\theta)d\theta $$
the mean value of $\tilde{\varphi}_{\eps,2}$ on half circles. It is a radial function. We also set 
$$ \psi_{\eps,2}(x) = \tilde{\varphi}_{\eps,2}(x) - m_{\eps,2}(\left\vert x \right\vert) \hskip.1cm.$$
We have
\begin{equation}
\label{eqsteklovtonpsieps2}
\begin{cases}
\Delta \psi_{\eps,2} = - \Delta m_{\eps,2} = -\frac{\sigma_{\eps,2}}{\pi r}e^{\tilde{\omega}_{\eps}}\left(\tilde{\varphi}_{\eps,2}(r,0)+\tilde{\varphi}_{\eps,2}(-r,0)\right) & \text{ in } \mathbb{R}^2_+ \\
-\partial_{x_2} \psi_{\eps,2} = \sigma_{\eps,2}e^{\tilde{\omega}_{\eps}}\tilde{\varphi}_{\eps,2} & \text{ on } \mathbb{R}\times\{0\} \hskip.1cm.
\end{cases}
\end{equation}
As for the first eigenfunction, we easily prove that $\psi_{\eps,2}$ is uniformly bounded so that we have for any $\rho>0$
$$ \int_{\rho\eps}^1 \left(\psi_{\eps,2}\right)^2 e^{\tilde{\omega}_{\eps}} \leq \frac{C}{\rho}$$
as $\eps\to 0$. 
Integrating once the equation on $m_{\eps,2}$ \eqref{eqsteklovtonpsieps2}, we have 
\begin{equation*}
\begin{split} m_{\eps,2}'(r) = +\frac{\sigma_{\eps,2}}{r}\int_r^1 e^{\tilde{\omega_{\eps}}} \left(\tilde{\varphi}_{\eps,2}(s,0) + \tilde{\varphi}_{\eps,2}(-s,0)\right) \\ = -\frac{\sigma_{\eps,2}}{r} \int_0^r e^{\tilde{\omega}_{\eps}} \left(\tilde{\varphi}_{\eps,2}(s,0) + \tilde{\varphi}_{\eps,2}(-s,0)\right)
\end{split}
\end{equation*}
so that by the sign properties of $\tilde{\varphi}_{\eps,2}$, $m_{\eps,2}$ is a decreasing function. It realizes its maximum for $r=0$ and its minimum for $r=1$.

Now, we prove that $m_{\eps,2}$ is uniformly bounded. We have that 
\begin{equation}
\label{eqproveunifboundedsteklov}
\begin{split}
& m_{\eps,2}(r)  - m_{\eps,2}(1) = 2\int_{r}^1 \frac{1}{s}\left(\int_{s}^1 \sigma_{\eps,1}e^{\tilde{\omega}_{\eps}}(-\tilde{\varphi}_{\eps,2})dt \right)ds \\
\leq & -2\tilde{\varphi}_{\eps,2}(1,0) \sigma_{\eps,2} \int_r^1 \frac{1}{s} \left(\int_{s}^1 e^{\tilde{\omega}_{\eps}}dt\right)ds \\
\leq & -2\tilde{\varphi}_{\eps,2}(1,0) \sigma_{\eps,2}\left(\ln\frac{\eps}{r}\left(\arctan\frac{1}{\eps}-\arctan\frac{r}{\eps}\right) + \int_{\frac{r}{\eps}}^{\frac{1}{\eps}}\frac{2\ln u}{1+u^2}du +O(\eps)\right)
\end{split}
\end{equation}
by a straightforward integral computation. Notice also that elliptic estimates on \eqref{eqsteklovtonpsieps2} imply that
$$ \tilde{\varphi}_{\eps,2}(1,0) - m_{\eps,2}(1)  = \psi_{\eps,2}(1,0) = O\left(\eps\right) $$
as $\eps\to 0$. Then, choosing $r=\rho \eps$ in  \eqref{eqproveunifboundedsteklov}, we obtain
\begin{equation}
\label{eqproveunifboundedsteklovrho}
m_{\eps,2}(r) - m_{\eps,2}(1) 
\leq  -2m_{\eps,2}(1) \sigma_{\eps,2}\left(\ln\frac{1}{\rho}\left(\frac{\pi}{2}-\arctan\rho \right) \right)
\end{equation}
for any $\eps>0$ small enough. Then, for $\rho$ large enough too, 
\begin{equation}
\label{eqproveunifboundedsteklovrho2}
- m_{\eps,2}(1) \left(1- 2\sigma_{\eps,2}\left(\ln\frac{1}{\rho}\left(\frac{\pi}{2}-\arctan\rho \right) \right) \right)  
\leq  -m_{\eps,2}(\rho) 
\end{equation}
proves that $m_{\eps,2}$ is uniformly bounded. We then have for any $\rho>0$ that
$$ \int_{\rho\eps}^1 \left(m_{\eps,2}\right)^2 e^{\tilde{\omega}_{\eps}} \leq \frac{C}{\rho}$$
as $\eps\to 0$. We deduce from this estimate on $m_{\eps,2}$ and the similar one on $\psi_{\eps,2}$ that for any $\rho>0$,
$$ \int_{\rho\eps}^1 \left(\tilde{\varphi}_{\eps,2}\right)^2 e^{\tilde{\omega}_{\eps}} \leq \frac{C}{\rho}$$
as $\eps\to 0$. Letting $\eps\to 0$ and then $\rho\to 0$, remembering \eqref{eqL2valuessecondsteklov}, \eqref{eqorthogonalitysteklovwithsymmetry2} and \eqref{equnifconvpgieps2steklov}, we get that
$$ \int_{\mathbb{R}\times\{0\}} \left(\varphi_{\star,2}^+\right)^2 e^u dx = \frac{\pi}{2} 
\text{ and } 
\int_{\mathbb{R}\times\{0\}}\varphi_{\star,2}^+ e^u dx = 0 \hskip.1cm. $$
By orthogonality with the constant function, we must have $\sigma_{\star,2}=2$ and $\varphi_{\star,2}^+ \circ f_+$ is a first Steklov eigenfunction of the disk  that is
$$ \varphi_{\star,2}^+ = a_{\star,1} Z_1 +a_{\star,2} Z_2 \hskip.1cm,$$
where $Z_1 = Re(f_+^{-1})$ and $Z_2 = Im(f_+^{-1})$ and $\left(a_{\star,1} \right)^2 +\left(a_{\star,2}\right)^2=1$. By symmetry \eqref{eqsym}, $a_{\star,1}=1$ and $a_{\star,2}=0$.

\medskip

Now, we aim at estimating $\delta_{\eps} = \left\vert \sigma_{\eps,2}-1 \right\vert$. 

\medskip

We set 
$$ R_{\eps}(x)= \tilde{\varphi}_{\eps,2}(x) - \tilde{\varphi}_{\eps,2}(0)Z_1\left(\frac{x}{\eps}\right)  $$
so that $R_{\eps}$ satisfies $R_{\eps}(0)=0$ and the symmetry $R_{\eps}(x_1,x_2) = R_{\eps}(-x_1,x_2) $ for any $(x_1,x_2) \in \mathbb{D}_+$. We have
\begin{equation} \label{eqonRepssteklovsymcase}
\begin{cases}
\Delta R_{\eps} = 0 & \text{ in } \mathbb{D}_+ \\
-\partial_{x_2} R_{\eps} - e^{\tilde{u}_{\eps}}R_{\eps} = \left(\sigma_{\eps,2}-1\right)e^{\tilde{u}_{\eps}}\tilde{\varphi}_{\eps,2}+\sigma_{\eps,2}e^{\tilde{v}_{\eps}}\tilde{\varphi}_{\eps,2} & \text{ on } [-1,1]\times\{0\} \hskip.1cm.
\end{cases}
\end{equation}
Integrating this equation against $R_{\eps}$, we obtain 
\begin{equation}\label{eqW12estonRepssym}
\begin{split}
\int_{\mathbb{D}_+} \left\vert \nabla R_{\eps}\right\vert^2 = \int_{\mathbb{S}^1_+} R_{\eps}\partial_{\nu} R_{\eps} + \left(\sigma_{\eps,2}-1\right) \int_{[-1,1]\times\{0\}}e^{\tilde{u}_{\eps}} \tilde{\varphi}_{\eps,2}R_{\eps} \\
+ \sigma_{\eps}^2 \int_{[-1,1]\times\{0\}} e^{\tilde{v}_{\eps}} \tilde{\varphi}_{\eps,2} R_{\eps} \\
\leq  \left\| \varphi_{\eps,2} \right\|_{\infty} \left\| R_{\eps} \right\|_{\infty} \left( \int_{\mathbb{S}^1_+} \left\vert \left(\partial_{\nu} Z_1\left(\frac{x}{\eps}\right)\right) \right\vert + \left\| R_{\eps} \right\|_{\infty} +2 \left\vert \sigma_{\eps,2} -1 \right\vert + O\left(\eps\right) \right) \\
\leq  C \left\| R_{\eps} \right\|_{\infty} \left(\left\| R_{\eps} \right\|_{\infty} + \delta_{\eps} + \eps  \right)
\end{split}
\end{equation}
for some constant $C$ independant from $\eps$, where $\left\| . \right\|_{\infty}$ denotes the uniform norm in $\mathbb{S}^1_+$. Here we easily computed that $\left\vert \partial_{\nu}\left( Z_1\left(\frac{x}{\eps}\right)\right) \right\vert = O(\eps)$ uniformly on $\mathbb{S}^1_+$. Letting 
\begin{equation}\label{defalphaepssym} \alpha_{\eps} = \left\| R_{\eps} \right\|_{\infty} + \delta_{\eps} + \eps^2 \hskip.1cm,\end{equation}
we get that 
\begin{equation} \label{eqW12estonRepssym2}
\int_{\mathbb{D}_+} \left\vert \nabla R_{\eps}\right\vert^2 \leq O\left( \alpha_{\eps}^2 \right)
\end{equation}
as $\eps\to 0$. Letting $\hat{R_{\eps}}(x) = R_{\eps}(\eps x)$, we get from \eqref{eqonRepssteklovsymcase} that
\begin{equation}
\label{eqonRNepssym}
\begin{cases}
\Delta \hat{R}_{\eps}  =0 & \text{ in } \mathbb{D}_{\frac{1}{\eps}}^+ \\
-\partial_{x_2}\hat{R}_{\eps} - e^{u} \hat{R}_{\eps} = \left(\sigma_{\eps,2}-1\right)e^{ u}\hat{\varphi}_{\eps,2} + \sigma_{\eps,2} \frac{2\eps^2}{1+\eps^4 \left\vert x \right\vert^2} \hat{\varphi}_{\eps,2} & \text{ on } [-\frac{1}{\eps},\frac{1}{\eps}]\times\{0\} \hskip.1cm.
\end{cases}
\end{equation}
Dividing by $\alpha_{\eps}$, by standard elliptic theory, we obtain up to the extraction of a subsequence that for any $\rho>0$,
\begin{equation} \label{defRNstarsym} \frac{\hat{R}_{\eps}}{\alpha_{\eps}} \to \hat{R}_{\star} \hbox{ in } \mathcal{C}^{2}\left(\mathbb{D}_\rho^+\right) \end{equation}
as $\eps\to 0$. We also have up to the extraction of a subsequence that $\frac{\delta_{\eps}}{\alpha_{\eps}} \to \delta_{\star}$ as $\eps\to 0$ and $\frac{\eps^2}{\alpha_{\eps}}\to 0$ as $\eps\to 0$. Letting $\eps\to 0$ in \eqref{eqonRNepssym}, we deduce
\begin{equation}
\label{eqonhatRstarsym}
\begin{cases}
\Delta \hat{R}_{\star}  =0 & \text{ in } \mathbb{R}^2_+ \\
-\partial_{x_2}\hat{R}_{\star} - e^{u} \hat{R}_{\star} = \delta_{\star}e^u Z_1 & \text{ on } \mathbb{R}\times\{0\} \hskip.1cm.
\end{cases}
\end{equation}
Setting $R_{\star} = \hat{R}_{\star}\circ f_+$ in $\mathbb{D}\setminus \{(-1,0)\}$, we obtain
\begin{equation}
\label{eqonRstarsym}
\begin{cases}
\Delta R_{\star}  =0 & \text{ in } \mathbb{D} \\
-\partial_{r} R_{\star} - R_{\star} = \delta_{\star}z_1 & \text{ on } \mathbb{S}^1\setminus \{(-1,0)\} \hskip.1cm,
\end{cases}
\end{equation}
By \eqref{eqW12estonRepssym2}, we can extend $R_{\star}$ in $\mathbb{D}$ so that \eqref{eqonRstarsym} is satisfied in $\mathbb{D}$. We integrate \eqref{eqonhatRstarsym} against $z_1$ and we immediatly get $\delta_{\star} =0$. By definition of $\delta_{\star}$, $\delta_{\eps}= o(\alpha_{\eps})$ as $\eps\to 0$ and we obtain by \eqref{defalphaepssym}:
\begin{equation} \label{eqestimateondeltaepssteklov} \delta_{\eps} = o\left( \left\| R_{\eps} \right\|_{\infty} +  \eps \right) 
\end{equation}
as $\eps\to 0$. Moreover, we obtain that $\Delta R_{\star} = 0$ and $\partial_r R_{\star} = R_{\star}$. This means that $R_{\star}$ is a first eigenfunction in $\mathbb{D}$. Since $R_{\star}$ satisfies 
$$R_{\star}(z_1,z_2) = - R_{\star}(-z_1,z_2) = R_{\star}(z_1,-z_2) \text{ and } R_{\star}(-1,0)=0 \hskip.1cm,$$ 
we obtain that $R_{\star}=0$. We obtain that for any $\rho>0$
\begin{equation} \label{eqestimateonRepsReps} R_{\eps}(\rho \eps ) = o\left( \left\| R_{\eps} \right\|_{\infty} +  \eps \right) 
\end{equation}
as $\eps\to 0$.

\medskip

We prove from now to the end of the subsection that
\begin{equation} \label{eqestonnormRepssteklov}\left\| R_{\eps} \right\|_{\infty} = O\left(\eps\right) \end{equation}
as $\eps\to 0$. We split $R_{\eps}(x) = A_{\eps}(x)+M_{\eps}(\left\vert x \right\vert)$ such that
$$ M_{\eps}(r) = \frac{1}{\pi}\int_0^\pi R_{\eps}(r\cos\theta,r\sin\theta)d\theta = m_{\eps}(r) -\frac{1}{\pi}\int_0^\pi Z_1\left(\frac{(rcos\theta, r\sin\theta)}{\eps}\right) d\theta$$
is the mean value of $R_{\eps}$ on circles. From \eqref{eqestimateonRepsReps}, we know that 
\begin{equation}\psi_{\eps,2}(\rho \eps ) = o\left( \left\| R_{\eps} \right\|_{\infty} +  \eps \right) \text{ and } M_{\eps}(\rho \eps ) = o\left( \left\| R_{\eps} \right\|_{\infty} +  \eps \right) \end{equation}
as $\eps\to 0$ for any $\rho>0$. 

Let $0 \leq r_{\eps}\leq 1$ be such that $R_{\eps}(r_{\eps}) = \left\| R_{\eps} \right\|_{\infty} $. If $r_{\eps} = O(\eps)$, we deduce \eqref{eqestonnormRepssteklov} from \eqref{eqestimateonRepsReps}. 

Let's prove that uniformly on $x\in \mathbb{D}_2\setminus \mathbb{D}_{\frac{1}{2}}$,
\begin{equation}\label{eqconvergencepsieps2inftystekov} A_{\eps}(r_{\eps}x) = O\left(\frac{\eps}{r_{\eps}} \left\| R_{\eps} \right\|_{\infty} + r_{\eps} \eps \right) \text{ and } A_{\eps}(\eps x) = O\left(\eps \right) \end{equation}
as $\eps \to 0$. 

We have that $A_{\eps} = \psi_{\eps,2} - \mu_{\eps}$.
Similarly to equation \eqref{eqsteklovtonpsieps2} on $\psi_{\eps,2}$, we obtain
\begin{equation}
\label{eqsteklovtonpsieps3}
\begin{cases}
\begin{split}\Delta A_{\eps} = - \Delta M_{\eps} =& -\frac{\sigma_{\eps,2}}{\pi r}e^{\tilde{\omega}_{\eps}}\left(\tilde{\varphi}_{\eps}(r,0)+\tilde{\varphi}_{\eps}(-r,0)\right)  \\
& +  \frac{1}{\pi r}e^{\tilde{u}_{\eps}} \left(Z_1\left(\frac{(r,0)}{\eps}\right)+Z_1\left(\frac{(-r,0)}{\eps}\right)\right) 
\end{split}
& \text{ in } \mathbb{R}^2_+  \\
-\partial_{x_2} A_{\eps} = \sigma_{\eps,2}e^{\tilde{\omega}_{\eps}}\tilde{\varphi}_{\eps} - e^{\tilde{u}_{\eps}}Z_1\left(\frac{x}{\eps}\right)  & \text{ on } \mathbb{R}\times\{0\} \hskip.1cm.
\end{cases}
\end{equation}
we rewrite as
\begin{equation}
\label{eqsteklovtonpsieps4}
\begin{cases}
\begin{split}\Delta A_{\eps} =& -\frac{\sigma_{\eps,2}}{\pi r}e^{\tilde{\omega}_{\eps}}\left(R_{\eps}(r,0)+R_{\eps}(-r,0)\right)  \\
& +  \frac{1}{\pi r}\left((\sigma_{\eps,2}-1)e^{\tilde{u}_{\eps}} +\sigma_{\eps,2}e^{\tilde{v}_{\eps}} \right) \left(Z_1\left(\frac{(r,0)}{\eps}\right)+Z_1\left(\frac{(-r,0)}{\eps}\right)\right)  
\end{split}
 \\
-\partial_{x_2} A_{\eps} = \sigma_{\eps,2}e^{\tilde{\omega}_{\eps}}R_{\eps} +  (\sigma_{\eps,2}-1)e^{\tilde{u}_{\eps}}Z_1\left(\frac{x}{\eps}\right) + \sigma_{\eps,2}e^{\tilde{v}_{\eps}}Z_1\left(\frac{x}{\eps}\right)  
\end{cases}
\end{equation}
Similarly to arguments on \eqref{eqpsiepssteklov+} at the end of subsection \ref{subsectionantisymmetriccase}, working at the scales $r_{\eps}$ and $\eps$ give \eqref{eqconvergencepsieps2inftystekov}.

\medskip

Now, from the equation \eqref{eqonRepssteklovsymcase}, we deduce
\begin{equation}
\label{eqonMepssym}
\begin{split}
\Delta M_{\eps}  = \frac{e^{\tilde{u}_{\eps}}}{r} \left(R_{\eps}(r,0)+R_{\eps}(-r,0)\right) +
 \left(\sigma_{\eps,2}-1\right)\Delta m_{\eps} \\ + \frac{e^{\tilde{v}_{\eps}}}{r}\left(\tilde{\varphi}_{\eps,2}(r,0)+\tilde{\varphi}_{\eps,2}(-r,0) \right)\hskip.1cm.
 \end{split}
\end{equation}
Integrating on this equation, we have for $r\leq 1$:
\begin{equation}
\label{eqsymestonRepssteklov}
\begin{split}
M_{\eps}(r) &  - M_{\eps}(1) =  -\int_{1}^r \frac{1}{s}\left(\int_{1}^s \Delta M_{\eps} t dt\right) ds \\
= &  -\int_{1}^r \frac{1}{s}\left(\int_{1}^s  e^{\tilde{u}_{\eps}} \left(R_{\eps}(t,0)+R_{\eps}(-t,0) \right)dt\right) ds \\
& + \left(\sigma_{\eps,2} - 1\right)\left(m_{\eps}(r)-m_{\eps}(1)\right)\\
& -  \int_{1}^r \frac{1}{s} \left(\int_{1}^s e^{\tilde{v}_{\eps}}\left(\tilde{\varphi}_{\eps,2}(t,0)+\tilde{\varphi}_{\eps,2}(-t,0) \right) dt \right)ds  + (1-r)\left(R_{\eps}\right)'(1) 
\end{split}
\end{equation}
So that since $\left(R_{\eps}\right)'(1) = O\left(\eps\right)$, that 
$$ \int_{1}^r \frac{1}{s}\left(\int_{1}^s  e^{\tilde{u}_{\eps}}dt\right) ds = \ln\frac{\eps}{r}\left(\arctan\frac{1}{\eps}-\arctan\frac{r}{\eps}\right) +\int_{\frac{r}{\eps}}^{\frac{1}{\eps}}\frac{2\ln u}{1+u^2}du$$
that $m_{\eps}$ is uniformly bounded and by definition of $\delta_{\eps}$ we obtain
\begin{equation} \label{eqMeps}
\begin{split}
\left\vert M_{\eps}(r) - M_{\eps}(1) \right\vert \leq \left\| R_{\eps} \right\|_{\infty}\left(\ln\frac{\eps}{r}\left(\arctan\frac{1}{\eps}-\arctan\frac{r}{\eps}\right) 
+ \int_{\frac{r}{\eps}}^{\frac{1}{\eps}}\frac{2\ln u}{1+u^2}du \right) \\ + O (\delta_{\eps} + \eps)\hskip.1cm.
\end{split}
\end{equation}
In particular, for $r=r_{\eps}$, with \eqref{eqestimateondeltaepssteklov}, $R_{\eps}(r_{\eps}) = \left\| R_{\eps} \right\|_{\infty}$, $\eps=o(r_{\eps})$ and \eqref{eqconvergencepsieps2inftystekov}, we obtain that 
\begin{equation}\label{eqfinalsteklov1} \left\| R_{\eps} \right\|_{\infty} \leq \left\vert R_{\eps}(1)\right\vert +O(\eps)  \end{equation}
and applying again \eqref{eqMeps} for $r=\rho \eps$ with $\rho$ large enough, we obtain that 
\begin{equation}\label{eqfinalsteklov2} \left\vert R_{\eps}(1) \right\vert \leq 2\left\vert R_{\eps}(\rho\eps) \right\vert + O\left(\eps\right) \end{equation}
as $\eps\to 0$. By \eqref{eqestimateonRepsReps}, we obtain \eqref{eqestonnormRepssteklov}.

\medskip

We conclude that $\delta_{\eps} = o(\eps)$ from \eqref{eqestimateondeltaepssteklov} and  \eqref{eqestonnormRepssteklov}, and subsection \ref{subsectionsymmetriccase} is complete.

\bibliographystyle{alpha}
\bibliography{mybibfile}

\end{document}